\documentclass[11pt]{amsart}
\allowdisplaybreaks
\usepackage{amsrefs} 
\usepackage{fancyhdr}
\usepackage[titletoc]{appendix}

\usepackage{amsmath,amssymb,amsthm}
\usepackage{bm}
\usepackage[dvipsnames]{xcolor}
\usepackage{etoolbox,adjustbox}
\usepackage{abstract}

\usepackage[overload]{textcase}
\usepackage{multirow}
\usepackage{graphics,graphicx}
\usepackage{float,verbatim,array}
\usepackage{todonotes}
\usepackage{xfrac,enumitem,soul}
\usepackage{tikz,tikz-cd}
\usepackage[breaklinks=true]{hyperref}
\hypersetup{colorlinks,linkcolor={},citecolor={blue},urlcolor={red}}

\usepackage[top=1in, bottom=1in, left=1in, right=1in]{geometry}
\usepackage{color,xcolor}
\colorlet{BLUE}{blue}
\colorlet{RED}{red}
\colorlet{GRAY}{gray}
\colorlet{BROWN}{brown}
\definecolor{OliveGreen}{rgb}{0,0.6,0}

\counterwithout{equation}{section} 
\numberwithin{equation}{section}
\newtheorem{thm}{Theorem}[section]
\newtheorem{lmm}[thm]{Lemma}
\newtheorem{crl}[thm]{Corollary}
\newtheorem{prp}[thm]{Proposition}

\theoremstyle{definition}
\newtheorem{dfn}[thm]{Definition}
\newtheorem{rmk}[thm]{Remark}
\newtheorem{note}[thm]{Note}
\newtheorem{question}[thm]{Question}
\newtheorem{claim}{Claim}
\newtheorem{notation}{Notation}

\theoremstyle{definition}
\newtheorem*{thm*}{Theorem}
\newtheorem*{lmm*}{Lemma}
\newtheorem*{crl*}{Corollary}
\newtheorem*{MRI*}{Result 1}
\newtheorem*{MRII*}{Result 2}
\newtheorem*{PRT*}{Path representation theorem}
\newtheorem*{claim*}{Claim}

\DeclareMathOperator{\dist}{dist}
\DeclareMathOperator{\diam}{diam}

\title[On Quasiconvexity of Precompact-Subset Spaces]{On Quasiconvexity of Precompact-Subset Spaces}
\author{Earnest Akofor}
\address{\textnormal{Department of Mathematics and Computer Science, 
         Faculty of Science, University of Bamenda, 
         PO Box 39 Bambili, NW Region, Cameroon}}
\email{eakofor@gmail.com}


\begin{document}

\begingroup
\def\uppercasenonmath#1{} 
\let\MakeUppercase\relax 
\maketitle

\begin{abstract}    
\vspace{0.2cm}

\noindent Let $X$ be a metric space and $BCl(X)$ the collection of nonempty bounded closed subsets of $X$ as a metric space with respect to Hausdorff distance. We study both characterization and representation of Lipschitz paths in $BCl(X)$ in terms of Lipschitz paths in $X$ and in the completion of $X$. We show that a full characterization and representation is possible in any subspace $\mathcal{J}\subset BCl(X)$ that (i) consists of precompact subsets of $X$, (ii) contains the singletons $\{x\}$ for every $x\in X$, and (iii) satisfies $BCl(C)\subset\mathcal{J}$ for every $C\in\mathcal{J}$. When $X$ is geodesic, we investigate quasiconvexity of $\mathcal{J}$ for some instances of $\mathcal{J}$, especially when $\mathcal{J}$ consists of finite subsets of $X$.
\end{abstract}
\let\bforigdefault\bfdefault
\addtocontents{toc}{\let\string\bfdefault\string\mddefault}
\tableofcontents
\section{\textnormal{\bf Introduction}}\label{Intro}

\noindent If $X$ is a (topological) space and $A\subset X$ a subset, then the \textbf{closure} of $A$ in $X$ is denoted by $cl_X(A)$, and, a continuous map $\gamma:[0,1]\rightarrow X$ (respectively, $\gamma:D\subset[0,1]\rightarrow X$ with $0,1\in D$) is called a \textbf{path} (respectively, a \textbf{partial path}) in $X$ from $\gamma(0)$ to $\gamma(1)$, or \textbf{connecting} $\gamma(0)$ to $\gamma(1)$. If $\gamma$ is a path in a metric space, then its \textbf{length} (Definition \ref{PathLenDfn}) is denoted by $\ell(\gamma)$. If $X$ and $Y$ are metric spaces and $L\geq 0$, then a map $f:X\rightarrow Y$ is \textbf{$L$-Lipschitz} if $d(f(x),f(x'))\leq Ld(x,x')$, for all $x,x'\in X$. Let $X$ be a metric space and $\lambda\geq 1$. Given a path $\gamma:[0,1]\rightarrow X$, we say $\gamma$ is \textbf{rectifiable} if $\ell(\gamma)<\infty$, and we say $\gamma$ is a \textbf{$\lambda$-quasigeodesic} (or a \textbf{$\lambda$-quasiconvex path}) if $\gamma$ is $\lambda d(\gamma(0),\gamma(1))$-Lipschitz (or equivalently, if $\ell(\gamma)\leq\lambda d(\gamma(0),\gamma(1))$ by Corollary \ref{GeodLength}), where a $1$-quasigeodesic is called a \textbf{geodesic}. Accordingly, $X$ is a \textbf{$\lambda$-quasigeodesic space} (or a \textbf{$\lambda$-quasiconvex space}) if every two points of $X$ are connected by a $\lambda$-quasigeodesic in $X$. A $1$-quasiconvex space is naturally called a \textbf{geodesic space}.

Throughout the rest of the introduction, unless said otherwise, let $X=(X,d)$ be a metric space, with its \textbf{completion} denoted by $\widetilde{X}$ and the \textbf{closure} $cl_{\widetilde{X}}(A)$ in $\widetilde{X}$ of any set $A\subset X$ denoted by $\widetilde{A}$. As metric spaces with respect to Hausdorff distance $d_H$ (see equation (\ref{HaudDist})), consider the collection $BCl(X)$ of nonempty bounded closed subsets of $X$, the collection $PCl(X)$ of totally bounded (or ``precompact'') members of $BCl(X)$, the collection $K(X)$ of compact members of $PCl(X)$, the collection $FS(X)$ of finite members of $K(X)$, and the collection $FS_n(X)$ of members of $FS(X)$ with cardinality at most $n$. A subspace $\mathcal{J}\subset BCl(X)$ is a \textbf{stable covering subspace} if it is both \textbf{covering} in the sense that $\{x\}\in\mathcal{J}$, for all $x\in X$, and \textbf{stable} in the sense that $BCl(C)\subset\mathcal{J}$, for all $C\in\mathcal{J}$. We will refer to quasigeodesics in $BCl(X)$ as \textbf{Hausdorff quasigeodesics} (i.e., quasigeodesics in $BCl(X)$ with respect to $d_H$). By \textbf{representation} of a path $\gamma:[0,1]\rightarrow\mathcal{J}\subset BCl(X)$ in terms of paths $\{\gamma_r:[0,1]\rightarrow \widetilde{X}\}_{r\in \Gamma}$ we mean a pointwise expression of the form $\gamma(t)=cl_{\widetilde{X}}\{\gamma_r(t):r\in \Gamma\}\cap X$, for all $t\in[0,1]$.

A detailed study of quasiconvexity properties of a subspace $\mathcal{J}\subset BCl(X)$ is naturally expected to involve a search for general characterizations and representations of Hausdorff quasigeodesics (in $\mathcal{J}$ between any two sets $A,B\in\mathcal{J}$) in terms of Lipschitz paths in $X$. Moreover, it is easier to picture a path in $BCl(X)$ in terms of (paths in $\widetilde{X}$ viewed as partial/approximate) paths in $X$. We will now highlight our main results.

\begin{PRT*}[\textcolor{blue}{Theorem \ref{RepThmPCl}}]
Let $L\geq 0$, $\mathcal{J}\subset PCl(X)$ a stable covering subspace, and $A,B\in \mathcal{J}$. There exists an $L$-Lipschitz path $\gamma:[0,1]\rightarrow\mathcal{J}$ from $A$ to $B$ if and only if there exists a family $\{\gamma_r:[0,1]\rightarrow\widetilde{X}\}_{r\in \Gamma}$ of $L$-Lipschitz paths in $\widetilde{X}$, where $\gamma_r$ is a path from $\gamma_r(0)\in \widetilde{A}$ to $\gamma_r(1)\in \widetilde{B}$, such that (i) $R:=\{(\gamma_r(0),\gamma_r(1)):r\in\Gamma\}\subset \widetilde{A}\times \widetilde{B}$ is a densely complete relation (Definition \ref{DenseCompl}) and
(ii) $~cl_{\widetilde{X}}\{\gamma_r(t):r\in \Gamma\}\cap X\in \mathcal{J}$, for all $t\in[0,1]$. Moreover:
\begin{itemize}[leftmargin=0.7cm]
\item The family of paths can be chosen such that (iii) $~\gamma(t)=cl_{\widetilde{X}}\{\gamma_r(t):r\in \Gamma\}\cap X$, for all $t\in[0,1]$.
\item When $\mathcal{J}\subset K(X)\subset PCl(X)$, we can replace $\widetilde{X}$ with $X$ and choose the collection $\{\gamma_r:[0,1]\rightarrow X\}_{r\in\Gamma}$ in the representation (iii) to be maximal.
\end{itemize}
Especially, if $\gamma$ be a $\lambda$-quasigeodesic (resp., rectifiable path), set $L:=\lambda d_H(A,B)$ (resp., $L:=\ell(\gamma)$).
\end{PRT*}

This result, which is an existence criterion for paths in $PCl(X)$, becomes an existence criterion for paths in $BCl(X)$ if the metric space $X$ is compact or proper. For a compact metric space $Z$, the work of M\'emoli and Wan in \cite[Theorem 3]{MemoliWan2023} gives a characterization of geodesics in $BCl(Z)=K(Z)$. Our result in Theorem \ref{RepThmPCl} therefore generalizes this characterization in $K(X)$ to a characterization in $PCl(X)$. For the representation aspect of Theorem \ref{RepThmPCl}, similar results are given by the \emph{Castaing representation} of Lipschitz compact-valued multifunctions (see \cite[Th\'eor\`eme 5.4]{castaing1967} and \cite[Theorem 7.1]{Chist2004}).

When $X$ is geodesic or endowed with more structures, the abundance of geodesics in $X$ enables a straightforward construction of quasigeodesics in many subspaces of $BCl(X)$, as it has been done by Kovalev and Tyson in \cite[Theorem 2.1 and Corollary 2.2]{kovalevTyson2007}, by M\'emoli and Wan in \cite[Theorem 3.6]{MemoliWan2023}, and by Fox in \cite[Theorem 3.4]{fox2022} when $X$ is a normed space. In particular, the proof of \cite[Theorem 2.1]{kovalevTyson2007} shows that Lipschitz paths in $X$ are sufficient for quasigeodesics in $BCl(X)$. Our description of Lipschitz paths through Lemma \ref{QGeodSuff} and Theorem \ref{RepThmPCl} goes further to show that Lipschitz paths in $\widetilde{X}$ (or partial paths in $X$) are necessary for Lipschitz paths in $PCl(X)$. The question of whether or not Lipschitz paths in $\widetilde{X}$ are also necessary for Lipschitz paths in $BCl(X)$ is open and posed as Question \ref{CritQuest}.

A related result (for $\lambda>1$ instead of $\lambda\geq 1$) in Proposition \ref{QGeodEquiv2} seems to indicate that in Theorem \ref{RepThmPCl} it might be possible to replace $PCl(X)$ with a larger subspace of $BCl(X)$, especially when $X$ is geodesic. Our uncertainty here is expressed as Question \ref{RepsQuest}.

\begin{enumerate}[leftmargin=0.0cm]
\item[]\vspace{-0.3cm}
\item[] \textbf{Question 1:} Let $\lambda\geq 1$, $\mathcal{J}\subset PCl(X)$ a stable covering subspace, and $A,B\in\mathcal{J}$ (which are hypotheses of Theorem \ref{RepThmPCl}). If $X$ is geodesic, does it follow that $\mathcal{J}$ is $\lambda$-quasiconvex?
\item[]\vspace{-0.3cm}
\end{enumerate}

\noindent Question 1 has already been answered for some instances of $\mathcal{J}$, which we consider important enough to be reviewed from a new perspective. Borovikova, Ibragimov, and Yousefi showed in \cite[Theorem 4.1]{BorovEtal2010} that $FS_n(\mathbb{R})$ is $4^n$-quasiconvex. Based on our related earlier work in \cite{akofor2020,akofor2019}, we show that if $X$ is geodesic, then the following are true:
\begin{enumerate}[leftmargin=0.8cm]
\item[(i)] $FS_2(X)$ and $FS(X)$ are geodesic (Corollaries \ref{FSGeodBound} and \ref{FS_BCl_QConv}).
\item[(ii)] $FS_n(X)$ is not geodesic for $n\geq3$ (Corollary \ref{GeodFailCrl}).
\item[(iii)] $FS_n(X)$ is $2$-quasiconvex (Theorem \ref{QConvFSn}), improving the above $4^n$-quasiconvexity of $FS_n(\mathbb{R})$.
\item[(iv)] For $n\geq 3$, the integer $2$ is the smallest quasiconvexity constant for $FS_n(X)$ (Theorem \ref{QConvFSn}).
\end{enumerate}
\noindent If $X$ is geodesic, the results (ii)-(iii) above say that, for $n\geq 3$, $FS_n(X)$ is $2$-quasiconvex but not geodesic. Similarly, if $X$ is geodesic, then $BCl(X)$ is $\lambda$-quasiconvex for $\lambda>1$ (Corollary \ref{FS_BCl_QConv}) but need not be geodesic, as noted in \cite{kovalevTyson2007} following the work of Bryant in \cite{bryant1970}.

The rest of the paper is organised as follows. In Section \ref{MSpGeodesics}, we formalize notation and review rectifiable paths and quasigeodesics in metric spaces. In Section \ref{SSpGeodesics}, we present our main results on quasiconvexity of precompact-subset spaces. This is followed in Section \ref{FSSpGeodesics} by a concise review of our earlier work on quasiconvexity of finite-subset spaces. In Section \ref{QGeodQuests}, we ask a few questions on extendability of our characterization and representation criterion/criteria for Lipschitz paths (e.g., Hausdorff quasigeodesics) in subset spaces and on efficiency of such paths.

\section{\textnormal{\bf Notation and review of rectifiable paths in metric spaces}}\label{MSpGeodesics}

We begin this section with some conventions and definitions that will be used throughout this paper. The abbreviation ``resp.'' stands for ``respectively'', and used to display two or more separate statements in a simultaneous or parallel manner whenever it seems convenient to do so. Given a set $S$, its \textbf{cardinality} and \textbf{powerset} are denoted by $|S|$ and $\mathcal{P}(S)$ respectively.

Throughout the rest of this section, unless said otherwise, let $X=(X,d)$ be a metric space. Consider subsets $A,B\subset X$ and a point $x\in X$. The \textbf{diameter} of $A$ is $\diam(A):=\sup_{a,a'\in A}d(a,a')$. The \textbf{distance} between $x$ and $A$ is
\[
\textstyle\dist(x,A)=\dist^X(x,A):=\inf_{a\in A}d(x,a),
\]
and, between $A$ and $B$ is $\dist(A,B):=\inf_{a\in A}\dist(a,B)$. For $R>0$, the \textbf{open $R$-neighborhood} of $A$ (resp., of $x$) is
\[
N_R(A)=N_R^X(A):=\{x\in X:\dist(x,A)<R\}~~~~\big(\textrm{resp.,}~N_R(x)=N_R^X(x):=N_R^X(\{x\})\big),
\]
and the \textbf{closed $R$-neighborhood} of $A$ (resp., of $x$) is
\[
\overline{N}_R(A)=\overline{N}_R^X(A):=\{x\in X:\dist(x,A)\leq R\}~~~~\big(\textrm{resp.,}~\overline{N}_R(x)=\overline{N}_R^X(x):=\overline{N}_R^X(\{x\})\big).
\]
The \textbf{Hausdorff distance} between $A$ and $B$ is
\begin{align}
\label{HaudDist}\textstyle d_H(A,B):=\max\left\{\sup\limits_{a\in A}\dist(a,B),\sup\limits_{b\in B}\dist(b,A)\right\}=\inf\{r:A\cup B\subset \overline{N}_r(A)\cap \overline{N}_r(B)\}.
\end{align}
A set $A\subset X$ is \textbf{bounded} if there exist $x\in X$ and $R>0$ such that $A\subset N_R(x)$.

Let $Cl(X)$ denote the collection of nonempty closed subsets of $X$. In this paper, a \textbf{subset space} of $X$ is any subcollection $\mathcal{J}\subset Cl(X)$ on which $d_H$ takes finite values, making $\mathcal{J}=(\mathcal{J},d_H)$ a metric space. Subset spaces that will be relevant to us include the following:
\begin{enumerate}[leftmargin=1cm]
\item\label{SspItem1} \textbf{$C$-regulated subset space} of $X$ (for any $C\in Cl(X)$),~ $~\mathcal{H}(X;C):=\{A\in Cl(X):d_H(A,C)<\infty\}$, as introduced by Kovalev and Tyson in \cite{kovalevTyson2007}. If $C\not\in BCl(X)$, then $\mathcal{H}(X;C)$ has neither (nonempty) stable nor covering subspaces (as $\{x\}\not\in \mathcal{H}(X;C)$ for all $x\in X$).
\item\label{SspItem2} \textbf{Bounded-subset space} of $X$,~ $~BCl(X):=\{A\in Cl(X):A~\textrm{is bounded}\}$, where $\mathcal{H}(X;C)\cap BCl(X)\neq\emptyset$ if and only if $C\in BCl(X)$, if and only if $\mathcal{H}(X;C)=BCl(X)$.
\item\label{SspItem3} \textbf{Precompact-subset space} of $X$,~ $~PCl(X):=\{A\in BCl(X):A~\textrm{is precompact}\}$, where $A\subset X$ is \textbf{totally bounded} (or ``\textbf{precompact}'', see Notation \ref{TbPcTerm}) if for every $\varepsilon>0$ there exists a finite set $F\subset A$ such that $A\subset N_\varepsilon(F)$.
\item \textbf{Compact-subset space} of $X$,~ $~K(X):=\{A\in PCl(X):A~\textrm{is compact}\}$, where $A\subset X$ is \textbf{compact} if every open cover of $A$ has a finite subcover.
\item \textbf{Finite-subset space} of $X$,~ $~FS(X):=\{A\in K(X):|A|<\infty\}$.
\item \textbf{$n$th Finite-subset space} (or \textbf{$n$th symmetric product}) of $X$,~ $~FS_n(X):=\{A\in FS(X):|A|\leq n\}$, for any integer $n\geq 1$.
\end{enumerate}

We therefore have the following order of containment.
\[
FS_n(X)\subset FS(X)\subset K(X)\subset PCl(X)\subset BCl(X)\subset Cl(X).
\]

A \textbf{complete} metric space is one in which every Cauchy sequence converges. The \textbf{completion} of $X$ (i.e., the complete metric space containing $X$ as a dense subspace) will be denoted by $\widetilde{X}$, and if $A\subset X$, then $\widetilde{A}\subset\widetilde{X}$ will denote the closure $cl_{\widetilde{X}}(A)$ of $A$ in $\widetilde{X}$. A metric space is \textbf{proper} if every bounded closed subset of the space is compact.

\begin{notation}\label{TbPcTerm}
If $\mathcal{J}\subset Cl(X)$, then the expression $X\subset\mathcal{J}$ means $\{x\}\in\mathcal{J}$ for all $x\in X$.

In this paper, although ``precompact'' means ``totally bounded'' by item (\ref{SspItem3}) above, in the mathematics literature however, a set in a \emph{space} is often called \emph{precompact} if it has a \emph{compact closure}. Meanwhile, a \emph{metric space} is \emph{totally bounded} if and only if it has a \emph{compact completion} (because it is \emph{compact} if and only if it is \emph{complete} and \emph{totally bounded}). From these basic observations, it is apparent that for metric spaces, \emph{completion} can be viewed as a strengthened/extended version of \emph{closure}, which is our main reason (apart from notational convenience) for loosely referring to ``totally bounded'' as ``precompact''.
\end{notation}

\begin{dfn}[\textcolor{blue}{Path, Parametrization, Length, Rectifiable, Constant speed, Natural parametrization}]\label{PathLenDfn}
Let $X$ be a space and $[a,b]\subset\mathbb{R}$ a compact interval (where we will mostly assume $[a,b]=[0,1]$ for simplicity). A continuous map $\gamma:[a,b]\rightarrow X$ is called a \textbf{path} in $X$ from $\gamma(a)$ to $\gamma(b)$, or \textbf{connecting} $\gamma(a)$ to $\gamma(b)$. Given paths $\gamma,\eta:[a,b]\rightarrow X$, $\eta$ is a \textbf{parametrization} of $\gamma$, written $\eta\sim\gamma$, if $\eta(a)=\gamma(a)$, $\eta(b)=\gamma(b)$, and $\eta([a,b])=\gamma([a,b])$.

Let $X$ be a metric space and $\gamma:[a,b]\rightarrow X$ a path. The \textbf{length} of $\gamma$ is
\begin{equation*}
\ell(\gamma)~:=~\sup\big\{\ell_P(\gamma):P\subset [a,b]~~\textrm{a finite partition}\big\},
\end{equation*}
where $\ell_P(\gamma):=\sum_{i=1}^kd(\gamma(t_{i-1}),\gamma(t_i))$ is the length of $\gamma$ over $P=\{a=t_0<t_1<\cdots<t_k=b\}$. If $\ell(\gamma)<\infty$, we say $\gamma$ is \textbf{rectifiable}. A path $\gamma:[a,b]\rightarrow X$ has \textbf{constant speed} $c\geq 0$ if $\ell(\gamma|_{[t,t']})=c|t-t'|$, for all $t,t'\in[a,b]$. A path with constant speed $c=1$ is called a \textbf{natural parametrization}.
\end{dfn}

In Definition \ref{PathLenDfn}, $\ell(\gamma)$ does not depend on the way $\gamma$ is parameterized, i.e., if two paths $\gamma,\eta:[a,b]\rightarrow X$ satisfy $\gamma\sim\eta$, then $\ell(\gamma)=\ell(\eta)$. The converse is false, i.e., for arbitrary paths $\gamma,\eta:[a,b]\rightarrow X$,  $\ell(\gamma)=\ell(\eta)$ does not imply $\eta\sim\gamma$. Also, \emph{path length is additive} in the sense that, given a rectifiable path $\gamma:[a,b]\rightarrow X$ and $a\leq c\leq b$, we have
\[
\ell(\gamma)=\ell(\gamma|_{[a,c]})+\ell(\gamma|_{[c,b]}).
\]

\begin{lmm}[\textcolor{blue}{Burago and Ivanov: \cite[Proposition 2.5.9]{BBI}}]\label{ConstSpeedPar}
Let $\gamma:[0,1]\rightarrow X$ be a rectifiable path and $l:=\ell(\gamma)$. There exists a nondecreasing continuous map $\varphi:[0,1]\rightarrow [0,l]$ and a natural parametrization $\overline{\gamma}:[0,l]\rightarrow X$ such that $\gamma=\overline{\gamma}\circ\varphi:[0,1]\stackrel{\varphi}{\longrightarrow}[0,l]\stackrel{\overline{\gamma}}{\longrightarrow}X$. 
\end{lmm}

An important consequence of Lemma \ref{ConstSpeedPar} is the following.

\begin{crl}[\textcolor{blue}{Constant speed parametrization of a rectifiable path}]\label{ConstSpeedRP}
Let $\gamma:[0,1]\rightarrow X$ be a rectifiable path. There exists a path $\eta:[0,1]\rightarrow X$ such that $~\eta\sim\gamma~$ and $~\ell(\eta|_{[t,t']})=\ell(\eta)|t-t'|$.
\end{crl}

\begin{dfn}[\textcolor{blue}{Quasigeodesic, Quasiconvex space, Geodesic, Geodesic space}]\label{QsiGeoDfn}
Let $X$ be a metric space and $\lambda\geq1$. A path $\gamma:[0,1]\rightarrow X$ is a \textbf{$\lambda$-quasigeodesic} (or a \textbf{$\lambda$-quasiconvex path}) if
\[
\textstyle \ell(\gamma|_{[t,t']})\leq\lambda d(\gamma(0),\gamma(1))|t-t'|,~~~~\textrm{for all}~~~~t,t'\in[0,1],
\]
which is the case if and only if $\gamma$ is $\lambda d(\gamma(0),\gamma(1))$-Lipschitz (see Lemma \ref{QgeodCharLmm}). We say $X$ is a \textbf{$\lambda$-quasigeodesic space} (or a \textbf{$\lambda$-quasiconvex space}) if every two points $x,y\in X$ are connected by a $\lambda$-quasigeodesic in $X$ (or equivalently, by a path $\gamma$ in $X$ of length $\ell(\gamma)\leq\lambda d(x,y)$ per Corollary \ref{GeodLength}). A $1$-quasigeodesic is called a \textbf{geodesic}, and similarly, a $1$-quasiconvex space is called a \textbf{geodesic space}.
\end{dfn}

Note that a $\lambda$-quasigeodesic is called a \textbf{$\lambda$-quasiconvex path} by Hakobyan and Herron in \cite{hakobyan-herron2008}. According to Tyson and Wu in \cite{tyson-wu2005}, a quasigeodesic is differently defined to be a path that is a bi-Lipschitz embedding, while injectivity of the path is not required in our definition. An equivalent definition of a geodesic in terms of paths that are parameterized by arc length has been given by Papadopoulos in \cite[Definition 2.2.1]{papado2014}.

\begin{note}\label{ConstSpNt}
Henceforth, we will assume for convenience that every rectifiable path is parameterized to have constant speed, even though some of the results might depend only partially on this assumption.
\end{note}

An important consequence of Lemma \ref{ConstSpeedRP} is the following.

\begin{crl}\label{GeodLength}
A (constant speed) path $\gamma:[0,1]\rightarrow X$ is a $\lambda$-quasigeodesic if and only if $\ell(\gamma)\leq \lambda d(\gamma(0),\gamma(1))$. In particular, $\gamma$ is a geodesic if and only if $\ell(\gamma)=d(\gamma(0),\gamma(1))$.
\end{crl}

\begin{lmm}[\textcolor{blue}{Characterization and sufficient condition for quasigeodesics}]\label{QgeodCharLmm}
Let $X$ be a metric space, $\gamma:[0,1]\rightarrow X$ a path, and $\lambda,\lambda_1,...,\lambda_k\geq 1$. The following are true:
\begin{itemize}[leftmargin=0.8cm]
\item[(i)] $\gamma$ is a $\lambda$-quasigeodesic if and only if $d(\gamma(t),\gamma(t'))\leq\lambda d(\gamma(0),\gamma(1))|t-t'|$, for all $t,t'\in[0,1]$.
\item[(ii)] Suppose $[0,1]=\bigcup_{i=1}^k[a_i,b_i]$, where $b_i=a_{i+1}$, for all $i=1,...,k-1$. If $\gamma|_{[a_i,b_i]}$ is a $\lambda_i$-quasigeodesic (for all $i=1,...,k$) and $\lambda=\max_i\lambda_i{d(\gamma(a_i),\gamma(b_i))\over d(\gamma(0),\gamma(1))}$, then $\gamma$ is a $\lambda$-quasigeodesic.
\end{itemize}
\end{lmm}
\begin{proof}
{\flushleft (i)} If $\gamma$ is a $\lambda$-quasigeodesic, then $d(\gamma(t),\gamma(t'))\leq \ell(\gamma|_{[t,t']})\leq\lambda d(\gamma(0),\gamma(1))|t-t'|$, for all $t,t'\in[0,1]$. Conversely, if  $d(\gamma(t),\gamma(t'))\leq\lambda d(\gamma(0),\gamma(1))|t-t'|$, for all $t,t'\in[0,1]$, then
\begin{align*}
 \ell(\gamma|_{[t,t']})&\textstyle=\sup\{\ell_P(\gamma):\textrm{finite}~P\subset[t,t']\}
=\sup\{\sum_{i=1}^kd(\gamma(t_{i-1}),\gamma(t_i)):t_i\in[t,t']\}\\
&\textstyle \leq \sup\{\sum_{i=1}^k\lambda d(\gamma(0),\gamma(1))|t_{i-1}-t_i|:t_i\in[t,t']\}=\lambda d(\gamma(0),\gamma(1))|t-t'|,
\end{align*}
for all $t,t'\in[0,1]$.
{\flushleft (ii)} By direct computation using additivity of path length, we have
\begin{align*}
\ell(\gamma|_{[t,t']})&\textstyle=\sum_j \ell(\gamma|_{[t,t']\cap[a_j,b_j]})\leq \sum_j\lambda_jd(\gamma(a_j),\gamma(b_j))\big|[t,t']\cap[a_j,b_j]\big|\\
 &\textstyle\leq \max_j\lambda_jd(\gamma(a_j),\gamma(b_j))|t-t'|,
\end{align*}
for all $t,t'\in[0,1]$.
\end{proof}

The following result is an immediate consequence of Lemma \ref{ConstSpeedRP}, Note \ref{ConstSpNt}, and Lemma \ref{QgeodCharLmm}(i).

\begin{lmm}[\textcolor{blue}{Characterization of geodesics: {\cite[Section 2.2]{papado2014}}}]\label{GeodCharLmm}
Let $X$ be a metric space and $\gamma:[0,1]\rightarrow X$ a path. The following are equivalent. (i) $\gamma$ is a geodesic. (ii) $d(\gamma(t),\gamma(t'))\leq d(\gamma(0),\gamma(1))|t-t'|$, for all $t,t'\in[0,1]$. (iii) $d(\gamma(t),\gamma(t'))=d(\gamma(0),\gamma(1))|t-t'|$, for all $t,t'\in[0,1]$.
\end{lmm}

\begin{dfn}[\textcolor{blue}{Path of minimum length}]
Let $X$ be a metric space, $x,y\in X$, and $\mathcal{P}_{x,y}(X)$:=$\{$paths in $X$ from $x$ to $y\}\subset\mathcal{C}\big([0,1],X\big)$. A path $\gamma:[0,1]\rightarrow X$ is said to \textbf{have minimum length} (or called a \textbf{path of minimum length}) if ~$\ell(\gamma)=\inf\left\{\ell(\eta)~|~\eta\in\mathcal{P}_{\gamma(0),\gamma(1)}(X)\right\}$.
\end{dfn}

\begin{dfn}[\textcolor{blue}{Length space, Rectifiably connected, Minimally connected}]
A metric space $(X,d)$ is called a \textbf{length space} if $d(x,y)=\inf\left\{\ell(\gamma)~|~\gamma\in\mathcal{P}_{x,y}(X)\right\}$, for all $x,y\in X$. Let us call a metric space $X$ \textbf{rectifiably connected} (resp., \textbf{minimally connected}) if every two points of $X$ can be connected by a rectifiable path (resp., a path of minimum length).
\end{dfn}

\begin{rmk}
If $(X,d)$ is rectifiably connected (resp., minimally connected), then $(X,d_0)$ is a length space (resp., a geodesic space), where $~d_0(x,y):=\inf\left\{\ell(\gamma):\gamma\in\mathcal{P}_{x,y}(X)\right\}$.
\end{rmk}

\section{\textnormal{\bf The case of precompact-subset spaces}}\label{SSpGeodesics}

\begin{dfn}[\textcolor{blue}{Relation: Left-complete, Right-complete, Complete, Reduced, Reduced complete, Proximal}]\label{RelTypeDfn}
Given sets $A$ and $B$, a \textbf{relation} $R\subset A\times B$ between $A$ and $B$ is called:
\begin{enumerate}[leftmargin=0.8cm]
\item[1.] \textbf{left-complete} if $~$for every $a\in A$, $~|(\{a\}\times B)\cap R|\geq 1$ (i.e., $(\{a\}\times B)\cap R\neq\emptyset$).
\item[2.] \textbf{right-complete} if $~$for every $b\in B$, $~|(A\times\{b\})\cap R|\geq 1$ (i.e., $(A\times\{b\})\cap R\neq\emptyset$).
\item[3.] \textbf{complete} if it is both left-complete and right-complete.
\item[4.] \textbf{reduced} if $~$for every $(a,b)\in R$, $~|(\{a\}\times B)\cap R|\leq 1$ or $|(A\times\{b\})\cap R|\leq 1$.
\item[5.] \textbf{reduced complete} if complete and for all $(a,b)\in R$, $~|(\{a\}\times B)\cap R|=1$ or $|(A\times\{b\})\cap R|=1$.
\item[6.] \textbf{$\lambda$-proximal} (where $\lambda\geq 1$) if $~$$\sup_{(a,b)\in R}d(a,b)\leq \lambda d_H(A,B)$.
\item[7.] \textbf{proximal} if it is $1$-proximal.
\end{enumerate}
Another relevant property, \textbf{dense completeness}, of a relation between subsets of a space is given in Definition \ref{DenseCompl}.
\end{dfn}

\begin{rmk}\label{CautionRmk}
Let $X$ be a metric space. Fix $\varepsilon>0$. If $x\in X$ and $A\in PCl(X)$, then $\dist(x,A)=\dist(x,\widetilde{A})=d(x,a')$ for some $a'\in\widetilde{A}\subset\widetilde{X}$ (where we know $\widetilde{A}$ is compact). Therefore,\\
(i) if $A,B\in PCl(X)$, then for each $a\in A$ there exists $b'=b'_a\in\widetilde{B}\subset\widetilde{X}$ such that
\[
d(a,b')=\dist(a,\widetilde{B})=\dist(a,B)\leq d_H(A,B)=d_H(\widetilde{A},\widetilde{B}).
\]
In general, if $A,B\in BCl(X)$, then for any $a\in A$ there exists $b_\varepsilon=b_{\varepsilon,a}\in B$ such that
\[
d(a,b_\varepsilon)<\dist(a,B)+\varepsilon\leq d_H(A,B)+\varepsilon = (1+\varepsilon_1)d_H(A,B),
\]
where the equality assumes that $d_H(A,B)\neq0$ and $\varepsilon_1=\varepsilon/d_H(A,B)$. In other words, if $\lambda>1$, then for any $a\in A$ there exists $b_\lambda=b_{\lambda,a}\in B$ such that ~$d(a,b_\lambda)\leq \lambda d_H(A,B)$. Therefore,\\
(ii) if $\lambda>1$, then we get a complete relation
\[
R_\lambda=\{(a,b)\in A\times B:d(a,b)\leq \lambda d_H(A,B)\}\subset A\times B,
\]
where if $d_H(A,B)=0$ (hence $A=B$), then $~R_\lambda=\{(a,a):a\in A\}\subset A\times A$.
\end{rmk}

\begin{dfn}
If $Z$ is a space and $\mathcal{A}\subset\mathcal{P}(Z)$, we denote by $\bigcup\mathcal{A}$ the union $\bigcup_{A\in\mathcal{A}}A$. We call a space (resp., metric space) $X$ \textbf{sequentially compact} (resp., \textbf{sequentially precompact}) if every sequence in $X$ has a convergent (resp., Cauchy) subsequence.
\end{dfn}

\begin{lmm}[\textcolor{blue}{Bounded union}]\label{BndImgLmm}
Let $X$ be a metric space and $X\subset\mathcal{J}\subset BCl(X)$. If $\mathcal{C}\subset\mathcal{J}$ is bounded, then $\bigcup\mathcal{C}$ is bounded in $X$.
\end{lmm}
\begin{proof}
Fix $x\in X$. Since $\mathcal{C}\subset \mathcal{J}$ is bounded, there exists $R>0$ such that $\mathcal{C}\subset N_R^\mathcal{J}(\{x\})=\{C\in \mathcal{J}:d_H(\{x\},C)<R\}=\{C\in \mathcal{J}:C\subset N_R^X(x)\}=\mathcal{J}\cap BCl\big(N_R^X(x)\big)$, i.e., $C\subset N_R^X(x)$ for all $C\in \mathcal{C}$, and so $\bigcup\mathcal{C}\subset N_R^X(x)$.
\end{proof}

\begin{lmm}[\textcolor{blue}{Precompact union, Compact union}]\label{UnionLmm}
Let $X$ be a metric space.

\noindent~(i) If $\mathcal{C} \subset PCl(X)$ is compact, then $K=\bigcup\mathcal{C}\subset X$ is precompact. In particular, $cl_X(K)\in PCl(X)$.

\noindent~(ii) If $\mathcal{C} \subset K(X)$ is compact, then $K=\bigcup \mathcal{C}\subset X$ is compact.
\end{lmm}
\begin{proof}
Fix $\mathcal{C}\subset PCl(X)$ and let $K=\bigcup\mathcal{C}$. Fix a sequence $\{x_k\}\subset K$. Then each $x_k \in C_k$ for some $C_k \in \mathcal{C}$. Since $\mathcal{C}$ is compact in $PCl(X)$, $\{C_k\}$ has a subsequence $\{C_{f(k)}\}$ that converges in $\mathcal{C}$. Let $C_{f(k)}\rightarrow C_0\in\mathcal{C}$. Since $C_0$ is precompact in $X$, by Remark \ref{CautionRmk}(i), there exist(s) $c_k\in \widetilde{C_0}\subset\widetilde{X}$ such that $d(x_{f(k)},c_k)=\dist(x_{f(k)},C_0)\leq d_H(C_{f(k)},C_0)\rightarrow 0$.

\noindent~\emph{Proof of (i):} Since $\widetilde{C_0}\subset\widetilde{X}$ is compact, $\{c_k\}\subset \widetilde{C_0}$ has a convergent (hence Cauchy) subsequence $c_{g(k)}$ (i.e., $d(c_{g(k)},c_{g(k')})\rightarrow 0$). Therefore $\{x_k\}\subset K$ has a Cauchy subsequence $\{x_{f\circ g(k)}\}$, since
\[
d(x_{f\circ g(k)},x_{f\circ g(k')})\leq d(x_{f\circ g(k)},c_{g(k)})+d(c_{g(k)},c_{g(k')})+d(c_{g(k')},x_{f\circ g(k')})\rightarrow 0.
\]
This shows that $K$ is precompact in $X$. Moreover, the closure of a precompact set is precompact.

\noindent~\emph{Proof of (ii):} Since $C_0\subset X$ is compact, $\{c_k\}\subset\widetilde{C_0}=C_0$ has a convergent subsequence $c_{g(k)}$. Therefore, with $c_{g(k)}\rightarrow c_0\in C_0$, we see that $\{x_k\}\subset K$ has a convergent subsequence $\{x_{f\circ g(k)}\}$ (hence $K$ is compact in $X$), since
$~d(x_{f\circ g(k)},c_0)\leq d(x_{f\circ g(k)},c_{g(k)})+d(c_{g(k)},c_0)\rightarrow 0$. \qedhere
\end{proof}

\begin{lmm}[\textcolor{blue}{Subsequence: pointwise Cauchy, pointwise convergent}]\label{PWCCLmm}
Let $T$ be a countable set.

\noindent~(i) If $K$ is a sequentially precompact space (e.g., a precompact metric space), then every sequence of maps $f_k : T \rightarrow K$ has a pointwise Cauchy subsequence $f_{s(k)} : T \rightarrow K$.

\noindent~(ii) If $K$ is a sequentially compact space (e.g., a compact metric space), then every sequence of maps $f_k : T \rightarrow K$ has a pointwise convergent subsequence $f_{s(k)} : T \rightarrow K$.
\end{lmm}
\begin{proof}
We prove (i) only, noting that the proof of (ii) is similar. Since $K$ is sequentially precompact, for each $t\in T$, the sequence $\{f_k(t)\}$ has a Cauchy subsequence. Consider an enumeration $T=\{t_1,t_2,\cdots\}$. Then $\left\{f_k(t_1)\right\}$ has a Cauchy subsequence $\left\{f_{s_1(k)}(t_1)\right\}$, i.e., there exists a subsequence $S_1=\left\{f_{s_1(k)}\right\}\subset\{f_k\}$ such that $\left\{f_{s_1(k)}(t_1)\right\}$ is Cauchy. Similarly, because $\left\{f_{s_1(k)}(t_2)\right\}$ has a Cauchy subsequence, there exists a further subsequence $S_2=\left\{f_{s_2(k)}\right\}\subset \left\{f_{s_1(k)}\right\}\subset\{f_k\}$ such that $f_{s_2(k)}(t_2)$ is Cauchy. Continuing this way, we get subsequences $S_1\supset S_2\supset S_3\supset\cdots$ of $\{f_k\}$, where $S_i=\{f_{s_i(k)}\}$ converges pointwise on $\{t_1,\cdots,t_i\}$.
Therefore the diagonal sequence ~$S=\left\{f_{s_k(k)}\right\}$ of $\{f_k\}$ converges pointwise on $T$. Define $f_{s(k)}:=f_{s_k(k)}$.
\end{proof}


\begin{thm}[\textcolor{blue}{Components of a Lipschitz path in $K(X)$}]\label{QGeodComp}
Let $X$ be a metric space, $L\geq 0$, and $\mathcal{J}\subset K(X)$. If $\gamma:[0,1]\rightarrow \mathcal{J}$ is an $L$-Lipschitz path in $\mathcal{J}$, then every $a\in \gamma(0)$ is connected to some $b\in \gamma(1)$ by an $L$-Lipschitz path $\gamma_{(a,b)}:[0,1]\rightarrow X$ such that $\gamma_{(a,b)}(t)\in\gamma(t)$, for all $t\in[0,1]$.
\end{thm}
\begin{proof}
Let $\gamma:[0,1]\rightarrow \mathcal{J}$ be an $L$-Lipschitz path from $A$ to $B$. Then $\gamma(0)=A$, $\gamma(1)=B$, ~$\gamma(t)\in\mathcal{J}$, for all $t\in[0,1]$, ~and
\[
d_H(\gamma(t),\gamma(t'))=\max\left\{\sup_{u\in\gamma(t)}\inf_{u'\in\gamma(t')}d(u,u'),\sup_{u'\in\gamma(t')}\inf_{u\in\gamma(t)}d(u,u')\right\}\leq{L}|t-t'|,
\]
for all $t,t'\in[0,1]$. For fixed $t,t'\in[0,1]$, this equation says for every $u\in\gamma(t)$ there exists $u'\in \gamma(t')$ such that
\begin{align}
\label{QGeodCompEq1} d(u,u')\leq d_H(\gamma(t),\gamma(t'))\leq {L}|t-t'|,~~~~\textrm{and vice versa}.
\end{align}
Let $D_k:=\{0=t_0<t_1<\cdots<t_k=1\}$, $k\geq 1$, be partitions of $[0,1]$ such that $D_k\subset D_{k+1}$ and $\bigcup_{k=1}^\infty D_k$ is dense in $[0,1]$ (e.g., $D_k=\{l/2^k:0\leq l\leq 2^k\}$). Fix $k\geq 1$. Then for each $a\in A$, we can define a map $g_k:D_k\rightarrow X$ as follows. Let $g_k(t_0)=g_k(0):=a\in A=\gamma(0)$. Next, using (\ref{QGeodCompEq1}), pick $g_k(t_1)\in\gamma(t_1)$ such that $d(g_k(t_0),g_k(t_1))\leq {L}|t_0-t_1|$. Inductively, given $g_k(t_i)\in\gamma(t_i)$, pick $g_k(t_{i+1})\in\gamma(t_{i+1})$ such that $d(g_k(t_i),g_k(t_{i+1}))\leq{L}|t_i-t_{i+1}|$. This gives an ${L}$-Lipschitz map $g_k:D_k\rightarrow X$ from $a\in A$ to $b_k=g_k(1)\in B$.

Consider the dense set $D:=\bigcup_{k=1}^\infty D_k$ in $[0,1]$. For each $k$, let $f_k:D\rightarrow X$ be any extension of $g_k:D_k\rightarrow X$ such that $f_k(t)\in\gamma(t)$, for all $t\in D$. Then $f_k(D)\subset K:=\bigcup_{t\in[0,1]}\gamma(t)$. Since $K\subset X$ is compact (Lemma \ref{UnionLmm}(ii)) and $D$ is countable, it follows by Lemma \ref{PWCCLmm} that $\{f_k\}$ has a pointwise convergent subsequence $\left\{f_{s(k)}\right\}$, where we know $f_{s(k)}$ is ${L}$-Lipschitz on $D_{s(k)}$. Let ${f}_{s(k)}\rightarrow f$ pointwise in ${X}$. Then $f:D\rightarrow {X}$ is ${L}$-Lipschitz. Also, $f(t)\in{\gamma(t)}$, for all $t\in D$, because
\begin{align}
&\dist(f(t),\gamma(t))\leq d\left(f(t),{f}_{s(k)}(t)\right)+\dist\left({f}_{s(k)}(t),\gamma(t)\right)=d\left(f(t),{f}_{s(k)}(t)\right)\rightarrow0,~~~~\textrm{for all}~~t\in D.\nonumber
\end{align}
Since $f$ is ${L}$-Lipschitz, and $D$ is dense in $[0,1]$, $f$ extends to an ${L}$-Lipschitz map $c:[0,1]\rightarrow {X}$. It remains to show that $c(t)\in{\gamma(t)}$, for all $t\in[0,1]$.

Fix $t\in[0,1]$. Since $D$ is dense in $[0,1]$, pick $t_j\in D$ such that $t_j\rightarrow t$. Then
\begin{align}
\dist(c(t),\gamma(t))&\leq d(c(t),c(t_j))+\dist(c(t_j),\gamma(t))=d(c(t),c(t_j))+\dist(f(t_j),\gamma(t))\nonumber\\
&\leq d(c(t),c(t_j))+d_H\big(\gamma(t_j),\gamma(t)\big)\rightarrow 0,\nonumber
\end{align}
from which it follows that $c(t)\in{\gamma(t)}$, for all $t\in [0,1]$.
\end{proof}

\begin{rmk}
In Theorem \ref{QGeodComp} we have mentioned ``components of a Lipschitz path in $K(X)$''. To make this precise, we will show in Theorem \ref{QGeodEquiv1} that the Lipschitz paths in $X$ obtained using Theorem \ref{QGeodComp} can be used to give a pointwise expression for Lipschitz paths in a stable covering subspace $\mathcal{J}\subset K(X)$.

Theorem \ref{QGeodComp} is closely related to the \emph{Lipschitz selection problem} studied in \cite{castaing1967}, \cite[Theorem 2]{hermes1971}, \cite[Theorem 1 and Question 2]{Slezak1987}, \cite[Theorem 3]{BelovChist2000}, \cite{Chist2004}, and the references therein.
\end{rmk}

\begin{dfn}[\textcolor{blue}{Densely complete relation, Subset connectivity, Weak subset connectivity}]\label{DenseCompl}~\\~
Let $X$ be a space and $A,B\subset X$. A relation $R\subset A\times B$ is \textbf{densely complete} if its domain $A_1:=\{a\in A:|(\{a\}\times B)\cap R|\geq 1\}$ is dense in $A$ and its \textbf{image} (or \textbf{range}) $B_1:=\{b\in B:|(A\times\{b\})\cap R|\geq 1\}$ is dense in $B$.

Let $X$ be a metric space, $\mathcal{J}\subset BCl(X)$, and $A,B\in \mathcal{J}$. We say $A$ is \textbf{$L$-Lipschitz} \textbf{$X$-connected} to $B$ in $\mathcal{J}$ if: there exists (i) a family $\left\{\gamma_r:[0,1]\rightarrow X\right\}_{r\in\Gamma}$ of $L$-Lipschitz paths in $X$, where $\gamma_r$ is a path from $\gamma_r(0)\in A$ to $\gamma_r(1)\in B$, such that (ii)  $R:=\{(\gamma_r(0),\gamma_r(1)):r\in\Gamma\}\subset A\times B$ is a densely complete relation and (iii) $cl_X{\left\{\gamma_r(t):r\in \Gamma\right\}}\in\mathcal{J}$, for each $t\in[0,1]$.

Similarly, we say $A$ is \textbf{$L$-Lipschitz} \textbf{weakly} \textbf{$X$-connected} to $B$ in $\mathcal{J}$ if: there exists (i) a family $\big\{\gamma_r:[0,1]\rightarrow \widetilde{X}\big\}_{r\in\Gamma}$ of $L$-Lipschitz paths in $\widetilde{X}$, where $\gamma_r$ is a path from $\gamma_r(0)\in \widetilde{A}$ to $\gamma_r(1)\in \widetilde{B}$, such that (ii) $R:=\{(\gamma_r(0),\gamma_r(1)):r\in\Gamma\}\subset \widetilde{A}\times \widetilde{B}$ is a densely complete relation and (iii) $cl_{\widetilde{X}}{\left\{\gamma_r(t):r\in \Gamma\right\}}\cap X\in\mathcal{J}$, for each $t\in[0,1]$.

\end{dfn}

Note that $A$ is \textbf{$L$-Lipschitz} \textbf{$X$-connected} to $B$ in $\mathcal{J}$ if and only if there exists a family $\left\{\gamma_r:[0,1]\rightarrow X\right\}_{r\in\Gamma}$ of $L$-Lipschitz paths in $X$ giving a densely complete relation $R:=\{(\gamma_r(0),\gamma_r(1)):r\in\Gamma\}\subset A\times B$ and an $L$-Lipschitz path from $A$ to $B$ by the map
\[
\gamma:[0,1]\rightarrow \mathcal{J},~~t\mapsto cl_X{\left\{\gamma_r(t):r\in \Gamma\right\}}.
\]
Similarly, by Lemma \ref{QGeodSuff} (later), $A$ is \textbf{$L$-Lipschitz} \textbf{weakly} \textbf{$X$-connected} to $B$ in $\mathcal{J}$ if and only if there exists a family $\big\{\gamma_r:[0,1]\rightarrow \widetilde{X}\big\}_{r\in\Gamma}$ of $L$-Lipschitz paths in $\widetilde{X}$ giving a densely complete relation $R:=\{(\gamma_r(0),\gamma_r(1)):r\in\Gamma\}\textcolor{red}{}\subset \widetilde{A}\times \widetilde{B}$ and an $L$-Lipschitz path from $A$ to $B$ by the map
\begin{align}
\label{WeakConnMap}\gamma:[0,1]\rightarrow \mathcal{J},~~t\mapsto cl_{\widetilde{X}}{\left\{\gamma_r(t):r\in \Gamma\right\}}\cap X.
\end{align}

\begin{rmk}[\textcolor{blue}{Approximation of paths and closedness of $PCl(X)$}]\label{CEx3Rmk}~\\~
Suppose that $A,B\in BCl(X)$ can be written as (strictly) increasing unions $A=\bigcup^{\uparrow}A_n$ and $B=\bigcup^{\uparrow}B_n$ for $A_n,B_n\in PCl(X)$. Further suppose $d_H(A_n,A)\rightarrow 0$ and $d_H(B_n,B)\rightarrow 0$, where
    \[
    \textstyle d_H(A_n,A)=\max\{\sup_{a'\in A_n}\dist(a',A),\sup_{a\in A}\dist(A_n,a)\}=\sup_{a\in A}\dist(A_n,a),
    \]
the further supposition being essential because $A_n\uparrow A$ does not imply $d_H(A_n,A)\rightarrow 0$ (which can be seen using the counterexample where $X$ is $\mathbb{N}$ with the discrete metric $\delta(a, b) = 1$ whenever $a\neq b$, $A = X$, and $A_n =\{1, 2, . . . , n\}$).

Then we may consider a geodesic {\small $\gamma_n:[0,1]\rightarrow PCl(X)\subset BCl(X)$} between $A_n$ and $B_n$ to approximate a geodesic between $A$ and $B$. Note that $A_n\cap B_{n'}\subset A_{\max(n,n')}\cap B_{\max(n,n')}$ and {\small $|d_H(A,B)-d_H(A_n,B_n)|\leq d_H(A,A_n)+d_H(B,B_n)\rightarrow 0$}, and so
    \[
    \textstyle A\cap B=\bigcup_{n,n'}A_n\cap B_{n'}=\bigcup_nA_n\cap B_n~~\textrm{and}~~d_H(A,B)=\lim_n d_H(A_n,B_n).
    \]
However, note that by the properties of $d_H$, $PCl(X)$ is closed in $BCl(X)$, and so $A,B\in PCl(X)$.
\end{rmk}

\begin{dfn}[\textcolor{blue}{Locally Lipschitz map, Lipschitz neighborhood of a point}]\label{LocLipMap}
Fix $c\geq 0$. A map $f:X\rightarrow Y$ is \textbf{locally $c$-Lipschitz} if for each $x\in X$ there exists a neighborhood $N_{r_x}(x)$, $r_x=r_{x,f}>0$, such that
\[
d(f(x),f(z))\leq cd(x,z),~~~~\textrm{for all}~~~~z\in N_{r_x}(x).
\]
We will call $N_{r_x}(x)$ a \textbf{Lipschitz neighborhood} of $x$ with respect to $f$.
\end{dfn}

\begin{lmm}[\textcolor{blue}{Gluing lemma}]\label{GluLmm}
Fix $\lambda\geq 1$. Let $X$ be a $\lambda$-quasiconvex space and $Y$ a metric space. If a map $f:X\rightarrow Y$ is locally $c$-Lipschitz, then it is $\lambda c$-Lipschitz.
\end{lmm}
\begin{proof}
Fix $x,x'\in X$. Let $\gamma:[0,1]\rightarrow X$ be a $\lambda$-quasigeodesic from $x$ to $x'$. Then $f\circ\gamma:[0,1]\stackrel{\gamma}{\longrightarrow}X\stackrel{f}{\longrightarrow} Y$ is locally $\lambda cd(x,x')$-Lipschitz on $[0,1]$. Indeed, for any $t\in[0,1]$ there is $r_{\gamma(t)}>0$ such that
\[
d\big(f(\gamma(t)),f(z)\big)\leq cd\big(\gamma(t),z\big),~~~~\textrm{for all}~~~~z\in N_{r_{\gamma(t)}}\big(\gamma(t)\big),
\]
and so the neighborhood $U=\gamma^{-1}\big(N_{r_{\gamma(t)}}\big(\gamma(t)\big)\big)$ of $t$ in $[0,1]$ satisfies
\[
d\big(f(\gamma(t)),f(\gamma(s))\big)\leq cd\big(\gamma(t),\gamma(s)\big)\leq \lambda cd(x,x')|t-s|,~~~~\textrm{for all}~~s\in U.
\]
Since $[0,1]$ is compact, for any $t,t'\in [0,1]$ we can choose a  partition $P=\{t=t_0<t_1<\cdots<t_k=t'\}$ such that Lipschitz neighborhoods of the $t_i$'s cover $[t,t']$. So, for each $i\in\{1,...,k\}$ there is $s_i\in[t_{i-1},t_i]$ satisfying $d\big(f(\gamma(t_{i-1})),f(\gamma(s_i))\big)\leq\lambda cd(x,x')|t_{i-1}-s_i|$ and $d\big(f(\gamma(s_i)),f(\gamma(t_i))\big)\leq\lambda cd(x,x')|s_i-t_i|$. By the triangle inequality,
$d\big(f(\gamma(t_{i-1})),f(\gamma(t_i))\big)\leq\lambda cd(x,x')(|t_{i-1}-s_i|+|s_i-t_i|)=\lambda cd(x,x')|t_{i-1}-t_i|$, and so
\[
\textstyle d\big(f(\gamma(t)),f(\gamma(t'))\big)\leq \sum_{i=1}^kd\big(f(\gamma(t_{i-1})),f(\gamma(t_i))\big)\leq \lambda cd(x,x')\sum_{i=1}^k|t_{i-1}-t_i|=\lambda cd(x,x')|t-t'|.
\]
This shows $f\circ\gamma$ is $\lambda cd(x,x')$-Lipschitz. Hence ~$d\big(f(x),f(x')\big)\leq \lambda cd(x,x')$.
\end{proof}

\begin{lmm}\label{FactznProp}
For any $A,B,C,D\in BCl(X)$, we have {\small $~d_H(A\cup B,C\cup D)\leq \max\big(d_H(A,C),d_H(B,D)\big)$}.
\end{lmm}

\begin{rmk}\label{IndexSetrmk}
The index set $\Gamma$ in Theorem \ref{QGeodEquiv1} or Theorem \ref{RepThmPCl} need not be (cardinality equivalent to) a relation $\Gamma\subset A\times B$ in general (although this will be the case for finite subset spaces in Theorem \ref{QGeodExistFS}). To see this, consider the counterexample with $X=\mathbb{R}^2$ and the geodesic $\gamma:[0,1]\rightarrow K(X),~t\mapsto\{t\}\times[0,\min(t,1-t)]$ from $A=\{(0,0)\}$ to $B=\{(1,0)\}$. The only nonempty relation $R\subset A\times B$ contains a single element.
\end{rmk}

\begin{thm}[\textcolor{blue}{Representation of Lipschitz paths in $K(X)$}]\label{QGeodEquiv1}
Fix $L\geq0$. Let $X$ be a metric space, $\mathcal{J}\subset K(X)$ a stable covering subspace, and $A,B\in\mathcal{J}$. Suppose $\gamma:[0,1]\rightarrow \mathcal{J}$ is an $L$-Lipschitz path from $A$ to $B$. Then there exists a maximal collection $\{\gamma_r:[0,1]\rightarrow X\}_{r\in \Gamma}$ of $L$-Lipschitz paths in $X$, where $\gamma_r$ is a path from $\gamma_r(0)\in A$ to $\gamma_r(1)\in B$, such that (i) $R:=\{(\gamma_r(0),\gamma_r(1)):r\in\Gamma\}\subset A\times B$ is a densely complete relation and (ii) $\gamma(t)=cl_X\{\gamma_r(t):r\in \Gamma\}$, for all $t\in[0,1]$.
\end{thm}
\begin{proof}
Let $\mathcal{S}$ be the set of pairs $(\eta,\Delta)$, where $\Delta$ is a set and $\eta:[0,1]\rightarrow\mathcal{J}$ is an ${L}$-lipschitz path satisfying $\eta(t)=cl_X\{\eta_r(t):r\in \Delta\}\subset\gamma(t)$, for ${L}$-Lipschitz paths $\eta_r:[0,1]\rightarrow X$, and $R:=\{(\eta_r(0),\eta_r(1)):r\in\Delta\}\subset \eta(0)\times\eta(1)$ is a densely complete relation. Consider $\mathcal{S}$ to be the poset with ordering $(\eta^1,\Delta^1)\leq (\eta^2,\Delta^2)$ if $\eta^1(t)\subset\eta^2(t)$, for all $t\in[0,1]$, and $\Delta^1\subset \Delta^2$. Then $\mathcal{S}$ is nonempty by Theorem \ref{QGeodComp}. Let $\{(\eta^\alpha,\Delta^\alpha)\}$ be a \emph{chain} in $\mathcal{S}$ (i.e., a \emph{linearly ordered subset} of $\mathcal{S}$).

\begin{claim}\label{ConstrClaim}
The map $\eta:[0,1]\rightarrow \mathcal{J},~t\mapsto cl_X\big(\bigcup_\alpha\eta^\alpha(t)\big)$ and $\Delta:=\bigcup_\alpha\Delta^\alpha$ give an element $(\eta,\Delta)\in\mathcal{S}$.

\noindent\emph{Proof of Claim \ref{ConstrClaim}:} Observe that $\eta(t)=cl_X\big[\bigcup_\alpha\eta^\alpha(t)\big]$=$cl_X\big[\bigcup_\alpha cl_X\{\eta_r^\alpha(t):r\in\Delta^\alpha\}\big]$=$cl_X\big[\bigcup_\alpha \{\eta_r^\alpha(t):r\in\Delta^\alpha\}\big]$=$cl_X\{\eta_r(t):r\in\Delta$=$\bigcup_\alpha\Delta^\alpha\}$, where $\eta_r:[0,1]\rightarrow X$ is given by $\eta_r(t):=\eta_r^\alpha(t)$ if $t\in\Delta^\alpha$. Since each $R^\alpha:=\{(\eta_r^\alpha(0),\eta_r^\alpha(1)):r\in\Delta^\alpha\}\subset \eta^\alpha(0)\times\eta^\alpha(1)$ is densely complete (i.e., $cl_X\{\eta_r^\alpha(0):r\in\Delta^\alpha\}=\eta^\alpha(0)$ and $cl_X\{\eta_r^\alpha(1):r\in\Delta^\alpha\}=\eta^\alpha(1)$), so is
\begin{align*}
 R:=&\textstyle~\{(\eta_r(0),\eta_r(1)):r\in\Gamma=\bigcup_\alpha\Delta^\alpha\}\\
 =&\textstyle~\bigcup_\alpha\{(\eta_r(0),\eta_r(1)):r\in\Delta^\alpha\}
\\
=&\textstyle~\bigcup_\alpha\{(\eta_r^\alpha(0),\eta_r^\alpha(1)):r\in\Delta^\alpha\}\\
=&\textstyle~\bigcup_\alpha R^\alpha\\
\subset &\textstyle~ cl_X\big(\bigcup_\alpha \eta^\alpha(0)\big)\times cl_X\big(\bigcup_\alpha\eta^\alpha(1)\big),
\end{align*}
by the definition of $\eta(t)$. This completes the proof of Claim \ref{ConstrClaim}.
\end{claim}
\noindent The pair $(\eta,\Delta)$ is an upper bound of $\{(\eta^\alpha,\Delta^\alpha)\}$ in $\mathcal{S}$. Therefore, by Zorn's lemma, $\mathcal{S}$ has a maximal element $(\eta',\Delta')$.

Suppose there is $s\in[0,1]$ such that $\eta'(s)\neq \gamma(s)$. Pick $x_s\in\gamma(s)\backslash\eta'(s)$. Using Theorem \ref{QGeodComp} and Lemma \ref{GluLmm}, construct an ${L}$-Lipschitz path $\eta_{(a_s,b_s)}:[0,1]\rightarrow X$, $\eta_{(a_s,b_s)}(t)\in \gamma(t)$, from some $a_s\in A$ through $x_s$ to some $b_s\in B$. Let $\eta''(t):=\eta'(t)\cup \eta_{(a_s,b_s)}(t)$ and $\Delta'':=\Delta'\cup\{(a_s,b_s)\}$. Then $(\eta',\Delta')<(\eta'',\Delta'')\in\mathcal{S}$, which is a contradiction.
\end{proof}

\begin{lmm}[\textcolor{blue}{Sufficient condition for Lipschitz paths in $BCl(X)$}]\label{QGeodSuff}
Fix $L\geq 0$. Let $X$ be a metric space, $\mathcal{J}\subset BCl(X)$, and $A,B\in\mathcal{J}$. Suppose $A$ is $L$-Lipschitz weakly $X$-connected to $B$ in $\mathcal{J}$. Then there exists an $L$-Lipschitz path from $A$ to $B$ in $\mathcal{J}$. In particular, if $L:=\lambda d_H(A,B)$ for some $\lambda\geq 1$ (in the said connectivity), then there exists a $\lambda$-quasigeodesic from $A$ to $B$ in $\mathcal{J}$.

(\textbf{Note}: This result also holds for $\mathcal{J}\subset\mathcal{H}(X;C)$, $C\in Cl(X)$, as in the proofs of \cite[Theorem 2.1 and Corollary 2.2]{kovalevTyson2007}.)
\end{lmm}
\begin{proof}
Consider a family of $L$-Lipschitz paths $\big\{\gamma_r:[0,1]\rightarrow \widetilde{X},~\gamma_r(0)\in \widetilde{A},\gamma_r(1)\in \widetilde{B}\big\}_{r\in \Gamma}$ in $\widetilde{X}$ such that $R:=\{(\gamma_r(0),\gamma_r(1)):r\in\Gamma\}\subset \widetilde{A}\times \widetilde{B}$ is a densely complete relation and $cl_{\widetilde{X}}\left\{\gamma_r(t):r\in \Gamma\right\}\cap X\in\mathcal{J}$, for all $t\in[0,1]$. Then the map $\gamma:[0,1]\rightarrow \mathcal{J}$ given by $\gamma(t):=cl_{\widetilde{X}}\left\{\gamma_r(t):r\in \Gamma\right\}\cap X$ is an $L$-Lipschitz path in $\mathcal{J}$ from $A$ to $B$, since $\gamma(0)=A$, $\gamma(1)=B$, and
\begin{align*}
d_H\big(\gamma(t),\gamma(t')\big)&=\max\left\{\sup_{r\in \Gamma}\inf_{r'\in \Gamma}d\left(\gamma_r(t),\gamma_{r'}(t')\right),\sup_{r'\in \Gamma}\inf_{r\in \Gamma}d\left(\gamma_r(t),\gamma_{r'}(t')\right)\right\}\\
&\leq\sup_{r\in \Gamma}d\left(\gamma_r(t),\gamma_r(t')\right)\leq L|t-t'|,
\end{align*}
for all $t,t'\in[0,1]$.
\end{proof}

The following result (which uses Remark \ref{CautionRmk}) is related to \cite[Corollary 2.2]{kovalevTyson2007} and \cite[Theorem 3.5]{MemoliWan2023}.

\begin{crl}\label{FS_BCl_QConv}
(i) Fix $\sigma\geq 1$ and $\lambda>1$. If $X$ is $\sigma$-quasiconvex, then $BCl(X)$ is $\sigma\lambda$-quasiconvex.
(ii) Fix $\lambda\geq 1$. If $X$ is $\lambda$-quasiconvex, then $FS(X)$ is $\lambda$-quasiconvex. (iii) Fix $\lambda\geq 1$. If $X$ is proper and $\lambda$-quasiconvex, then $BCl(X)=K(X)$ is $\lambda$-quasiconvex.
\end{crl}
\begin{proof}
\noindent(i): Assume $X$ is $\sigma$-quasiconvex and let $A,B\in \mathcal{J}=BCl(X)$. By the definition of $d_H$ (see Remark \ref{CautionRmk}), the relation $R=\{(a,b)\in A\times B:d(a,b)\leq\lambda d_H(A,B)\}$ is complete. Therefore, by $\sigma$-quasiconvexity of $X$, we have a collection $\{\gamma_{(a,b)}:(a,b)\in R\}$ of $\sigma\lambda d_H(A,B)$-Lipschitz paths in $X$, where $\gamma_{(a,b)}$ is a path from $a$ to $b$. Since $cl_X\{\gamma_r(t):r\in R\}\in\mathcal{J}$, for all $t\in[0,1]$, it follows that $A$ is $\sigma\lambda d_H(A,B)$-Lipschitz $X$-connected to $B$ in $\mathcal{J}$.

\noindent~(ii) and (iii): If $A,B\in \mathcal{J}=FS(X)$ or $A,B\in \mathcal{J}=K(X)$, then the relation $R=\{(a,b)\in A\times B:d(a,b)\leq d_H(A,B)\}$ is complete. The rest of the argument is the same as for (i).
\end{proof}

\begin{lmm}\label{QGeodExist}
Fix $L\geq 0$. Let $X$ be a metric space, $\mathcal{J}\subset PCl(X)$ a stable covering subspace, and $A,B\in\mathcal{J}$. Let $\widetilde{\mathcal{J}}:=\big\{\{x\}:x\in\widetilde{X}\big\}\cup\bigcup_{C\in\mathcal{J}}BCl(\widetilde{C})\subset K(\widetilde{X})$, which is also a stable covering subspace. Then the following statements are equivalent:
\begin{enumerate}[leftmargin=0.9cm]
\item[(i)]\label{QGeodExistST1} There exists an $L$-Lipschitz path $\gamma:[0,1]\rightarrow\mathcal{J}$ from $A$ to $B$.
\item[(ii)]\label{QGeodExistST2} There exists an $L$-Lipschitz path $\widetilde{\gamma}:[0,1]\rightarrow\widetilde{\mathcal{J}}$ from $\widetilde{A}$ to $\widetilde{B}$, such that the map $\gamma:[0,1]\rightarrow\mathcal{J},~t\mapsto\widetilde{\gamma}(t)\cap X$ is an $L$-Lipschitz path from $A$ to $B$.
\item[(iii)]\label{QGeodExistST3} $\widetilde{A}$ is $L$-Lipschitz $\widetilde{X}$-connected to $\widetilde{B}$ in $\widetilde{\mathcal{J}}$, such that restricting the connectivity to $X$ implies $A$ is $L$-Lipschitz weakly $X$-connected to $B$ in $\mathcal{J}$.
\item[(iv)] $A$ is $L$-Lipschitz weakly $X$-connected to $B$ in $\mathcal{J}$.
\end{enumerate}
\end{lmm}
\begin{proof}
Let $A,B\in \mathcal{J}$.

{\flushleft (i) implies (ii)}: Let $\gamma:[0,1]\rightarrow\mathcal{J}$ be an $L$-Lipschitz path from $A$ to $B$. Then $\widetilde{\gamma}:[0,1]\rightarrow\widetilde{\mathcal{J}},~t\mapsto \widetilde{\gamma(t)}$ gives the desired $L$-Lipschitz path from $\widetilde{A}$ to $\widetilde{B}$.

{\flushleft (ii) implies (iii)}: Let $\widetilde{\gamma}:[0,1]\rightarrow \widetilde{\mathcal{J}}$ be an $L$-Lipschitz path from $\widetilde{A}$ to $\widetilde{B}$, such that the map $\gamma:[0,1]\rightarrow\mathcal{J},~t\mapsto\widetilde{\gamma}(t)\cap X$ is an $L$-Lipschitz path from $A$ to $B$. By Theorem \ref{QGeodEquiv1}, there exists a maximal collection $\{\widetilde{\gamma}_r:[0,1]\rightarrow \widetilde{X}\}_{r\in \Gamma}$ of $L$-Lipschitz paths such that $R:=\{(\widetilde{\gamma}_r(0),\widetilde{\gamma}_r(1)):r\in\Gamma\}\subset \widetilde{A}\times \widetilde{B}$ is a densely complete relation and $~\widetilde{\gamma}(t)=cl_{\widetilde{X}}\{\widetilde{\gamma}_r(t):r\in \Gamma\}$, for all $t\in[0,1]$. \label{footnote3}
Restriction, to $X$, of the resulting $L$-Lipschitz $\widetilde{X}$-connection between $\widetilde{A}$ and $\widetilde{B}$ in $\widetilde{\mathcal{J}}$ now yields the desired $L$-Lipschitz weak $X$-connection between $A$ and $B$ in $\mathcal{J}$.

\noindent~(iii) implies (iv): This is immediate by hypotheses.

\noindent~(iv) implies (i): If $A$ is $L$-Lipschitz weakly $X$-connected to $B$ in $\mathcal{J}$, then by Lemma \ref{QGeodSuff} we get an $L$-Lipschitz path from $A$ to $B$ in $\mathcal{J}$.
\end{proof}

\begin{rmk}\label{WeakConnRmk}
The weak connectivity (by means of paths in $\widetilde{X}$ instead of paths in $X$) in Lemma \ref{QGeodExist} is essential. To see this, consider the counterexample with $X=\mathbb{Q}\subset\mathbb{R}$ (having no nonconstant paths) and the geodesic $\gamma:[0,1]\rightarrow PCl(\mathbb{Q}),~t\mapsto[t,t+1]\cap\mathbb{Q}$ from $A=[0,1]\cap\mathbb{Q}$ to $B=[1,2]\cap\mathbb{Q}$, which cannot be represented by paths in $X$.
\end{rmk}

The following is our main result of this section.

\begin{thm}[\textcolor{blue}{Representation of Lipschitz paths in $PCl(X)$}]\label{RepThmPCl}
Fix $L\geq 0$. Let $X$ be a metric space, $\mathcal{J}\subset PCl(X)$ a stable covering subspace, and $A,B\in \mathcal{J}$. There exists an $L$-Lipschitz path $\gamma:[0,1]\rightarrow\mathcal{J}$ from $A$ to $B$ if and only if there exists a family $\{\gamma_r:[0,1]\rightarrow\widetilde{X}\}_{r\in \Gamma}$ of $L$-Lipschitz paths in $\widetilde{X}$, where $\gamma_r$ is a path from $\gamma_r(0)\in \widetilde{A}$ to $\gamma_r(1)\in \widetilde{B}$, such that (i) $R:=\{(\gamma_r(0),\gamma_r(1)):r\in\Gamma\}\subset \widetilde{A}\times \widetilde{B}$ is a densely complete relation and (ii) $~cl_{\widetilde{X}}\{\gamma_r(t):r\in \Gamma\}\cap X\in \mathcal{J}$, for all $t\in[0,1]$. Moreover:
\begin{itemize}[leftmargin=0.7cm]
\item The family of paths can be chosen such that (iii) $~\gamma(t)=cl_{\widetilde{X}}\{\gamma_r(t):r\in \Gamma\}\cap X$, for all $t\in[0,1]$.
\item From the proof of Lemma \ref{QGeodExist}, when $\mathcal{J}\subset K(X)\subset PCl(X)$, we can replace $\widetilde{X}$ with $X$ and choose the collection $\{\gamma_r:[0,1]\rightarrow X\}_{r\in\Gamma}$ in the representation (iii) to be maximal.
\end{itemize}
Especially, if $\gamma$ be a $\lambda$-quasigeodesic (resp., rectifiable path), set $L:=\lambda d_H(A,B)$ (resp., $L:=\ell(\gamma)$).
\end{thm}
\begin{proof}
The equivalence is simply a restatement of Lemma \ref{QGeodExist} as the equivalence ``(i) if and only if (iv)'' in the lemma. Conclusion-(iii) is obtained as follows.\\
\indent In the proof of Lemma \ref{QGeodExist}, given an $L$-Lipschitz path $\gamma:[0,1]\rightarrow\mathcal{J}$ from $A$ to $B$, by lifting it to $\widetilde{\gamma}:[0,1]\rightarrow\widetilde{\mathcal{J}},~t\mapsto\widetilde{\gamma(t)}$, we get its representation in terms of $L$-Lipschitz paths by means of the equality $\gamma(t)=\widetilde{\gamma(t)}\cap X$, for all $t\in[0,1]$, where $\widetilde{\gamma}$ has the representation obtained at step ``(ii) implies (iii)'' of the proof of Lemma \ref{QGeodExist}.
\end{proof}

\begin{rmk}
Let $X$ be a metric space. As an immediate consequence of Theorem \ref{RepThmPCl}, if $X$ contains no nonconstant rectifiable paths (e.g., when $X$ is a $p$-snowflake, detailed in \cite{tyson-wu2005}), then neither does $K(X)$. A related result is this: By the proof of \cite[Proposition 1.4.11 in Section 1.4]{akofor2020}, if $X$ is a $p$-snowflake (see \cite{tyson-wu2005}), then so is $BCl(X)$.
\end{rmk}

\begin{note}\label{QGeodNote}
Fix $\lambda\geq 1$. If $A\in BCl(X)$ and $R>0$, let $A_R:=\overline{N}_R(A)=\{x\in X:\dist(x,A)\leq R\}$. For any $\lambda$-quasigeodesic $\gamma:[0,1]\rightarrow BCl(X)$, with ${L}:=\lambda d_H(\gamma(0),\gamma(1))$, the definition of Hausdorff distance $d_H$ implies
\[
\gamma(t)\subset \gamma(0)_{t{L}}\cap\gamma(1)_{(1-t){L}},~~\textrm{for all}~~t\in[0,1],
\]
since $d_H(\gamma(0),\gamma(t))\leq {L} t$ and $d_H(\gamma(t),\gamma(1))\leq {L} (1-t)$. Moreover, any ${L}$-Lipschitz path $c:[0,1]\rightarrow X$ from any $a\in \gamma(0)$ to any $b\in\gamma(1)$ similarly satisfies
\[
c(t)\in a_{t{L}}\cap b_{(1-t){L}}\subset \gamma(0)_{t{L}}\cap\gamma(1)_{(1-t){L}},~~\textrm{for all}~~t\in[0,1].
\]
\end{note}

\begin{prp}[\textcolor{blue}{Representation of some quasigeodesics in BCl(X)}]\label{QGeodEquiv2}
Fix $\lambda>1$. Let $X$ be a geodesic space, $A,B\in BCl(X)$, and $L:=\lambda d_H(A,B)$. Suppose the map $\gamma:[0,1]\rightarrow BCl(X)$, $\gamma(t)=A_{tL}\cap B_{(1-t)L}$, is a $\lambda$-quasigeodesic (e.g., when $A,B\in K(X)$, as shown by M\'emoli and Wan in \cite[Theorem 3.6]{MemoliWan2023} or by Serra in \cite[Theorem 1]{serra1998}).

Then there exists a maximal collection $\{\gamma_r:[0,1]\rightarrow X\}_{r\in \Gamma}$ of $L$-Lipschitz paths in $X$, where $\gamma_r$ is a path from $\gamma_r(0)\in A$ to $\gamma_r(1)\in B$, such that (i) $R:=\{(\gamma_r(0),\gamma_r(1)):r\in\Gamma\}\subset A\times B$ is a densely complete relation and (ii) $\gamma(t)=cl_X\{\gamma_r(t):r\in \Gamma\}$, for all $t\in[0,1]$.
\end{prp}
\begin{proof}
This is just a version of the proof of Theorem \ref{QGeodEquiv1} based on Remark \ref{CautionRmk}, geodesy of $X$, Note \ref{QGeodNote}, Claim \ref{ConstrClaim}, and Lemma \ref{GluLmm}.
\end{proof}

\section{\textnormal{\bf The case of finite-subset spaces}}\label{FSSpGeodesics}

\noindent Since $FS_n(X)$ and $FS(X)$ are stable covering subspaces of $PCl(X)$, the results of Section \ref{SSpGeodesics} already say a lot about finite-subset spaces (see Corollary \ref{FS_BCl_QConv} for example). Our aim here is to further consider a few of the related properties of finite-subset spaces that are not yet obvious from the results of Section \ref{SSpGeodesics}.

For any space $X$, $FS_n(X)$ is a quotient space of $X^n$ via the ``unordering'' map $q:X^n\rightarrow FS_n(X)$, $(x_1,...,x_n)$$\mapsto$$\{x_1,...,x_n\}$ as a quotient map. Consequently, we will switch notation and write an element $x\in FS_n(X)$ in the form $x=\{x_1,...,x_n\}=q(x_1,...,x_n)$ for an element $(x_1,...,x_n)\in X^n$.

Since finite subsets are compact, we have the quasiconvexity constant $\lambda\geq 1$ throughout this section. Recall the meaning of complete, reduced, and proximal relations from Definition \ref{RelTypeDfn}.

The following result is the existential part of the statement of Theorem \ref{RepThmPCl} for the stable covering subspace $\mathcal{J}=FS_n(X)\subset K(X)\subset PCl(X)$. It is used here to illustrate applicability of Theorem \ref{RepThmPCl}, specifically, to show that $FS_n(X)$ is not geodesic for $n\geq 3$ (see the alternative proof of Corollary \ref{GeodFailCrl}(ii)).

\begin{thm}[\textcolor{blue}{Criterion for Lipschitz paths in $FS_n(X)$}]\label{QGeodExistFS}
Let $X$ be a metric space and $n\geq 1$. For any $x,y\in FS_n(X)$, an $L$-Lipschitz path $\gamma$ exists from $x$ to $y$ in $FS_n(X)$ if and only if there exist (i) a complete relation $R\subset x\times y$ and (ii) a collection $\left\{\gamma_{(a,b)}:(a,b)\in R\right\}$ of $L$-Lipschitz paths in $X$, where $\gamma_{(a,b)}$ is a path from $a$ to $b$, such that $\left|\left\{\gamma_r(t):r\in R\right\}\right|\leq n$, for all $t\in[0,1]$.

Especially, if $\gamma$ be a $\lambda$-quasigeodesic (resp., rectifiable path), set $L:=\lambda d_H(x,y)$ (resp., $L:=\ell(\gamma)$).
\end{thm}
\begin{proof}
See the proof of Theorem \ref{RepThmPCl} for the case of $\mathcal{J}=FS_n(X)$.
\end{proof}

Note that the relation $R\subset x\times y$ in Theorem \ref{QGeodExistFS} is necessarily $\lambda$-proximal if $L=\lambda d_H(x,y)$, since
\[
d(a,b)=d(\gamma_{(a,b)}(0),\gamma_{(a,b)}(1))\leq L|0-1|=\lambda d_H(x,y),
\]
for all $(a,b)\in R$.

The following is a restricted version of Lemma \ref{QGeodSuff}.

\begin{thm}[\textcolor{blue}{Sufficient condition for Lipschitz paths in $FS(X)$}]\label{QGeodSuffFS}
Let $X$ be a $\lambda$-quasiconvex space, $x,y\in FS_n(X)$, and $m\geq n$. If there exists an $\alpha$-proximal complete relation $R\subset x\times y$ such that $|R|\leq m$, then there exists a $\lambda\alpha$-quasigeodesic between $x$ and $y$ in $FS_m(X)$.
\end{thm}
\begin{proof}
Assume some $R\subset x\times y$ is an $\alpha$-proximal complete relation with $|R|\leq m$. Then the map $\gamma:[0,1]\rightarrow X(m)$ given by
$\gamma(t):=\left\{\gamma_{(a,b)}(t):(a,b)\in R,~\gamma_{(a,b)}~\textrm{a $\lambda$-quasigeodesic in $X$ from $a$ to $b$}\right\}$ is a $\lambda\alpha$-quasigeodesic from $x$ to $y$, since $\gamma(0)=x$, $\gamma(1)=y$, and, for all $t,t'\in[0,1]$,
\begin{align*}
d_H(\gamma(t),\gamma(t'))&=\max\Big\{\!\max_{(a,b)\in R}\min_{(c,d)\in R}d\left(\gamma_{(a,b)}(t),\gamma_{(c,d)}(t')\right),\max_{(c,d)\in R}\min_{(a,b)\in R}d\left(\gamma_{(a,b)}(t),\gamma_{(c,d)}(t')\right)\!\Big\}\\
   &\leq\textstyle\max_{(a,b)\in R}d\left(\gamma_{(a,b)}(t),\gamma_{(a,b)}(t')\right)\leq\lambda\max_{(a,b)\in R}d(a,b)|t-t'|\\
   &\leq\lambda\alpha d_H(x,y)|t-t'|. \qedhere
\end{align*}
\end{proof}

\begin{dfn}[\textcolor{blue}{Midpoint, Opposite points, Spaced pair}]
Let $X$ be a geodesic space and $x,y,z\in X$. We call $z$ a \textbf{midpoint} between $x$ and $y$ if ~$d(x,y)=d(x,z)+d(z,y)=2d(x,z)=2d(z,y)$. In this case, we also say $x$ and $y$ are \textbf{opposite} with respect to $z$. In a metric space $Z$, two points $z_1,z_2\in Z$ form a \textbf{$(\sigma,D(\sigma,z_1,z_2))$-spaced pair} (where $\sigma>0$ and $D(\sigma,z_1,z_2)\subset Z$) if $\overline{N}_r(z_1)\cap D(\sigma,z_1,z_2)\cap\overline{N}_r(z_2)=\emptyset$ for every $0<r<\sigma d(z_1,z_2)$, or equivalently, if $\sigma d(z_1,z_2)\leq\max\{d(z_1,z),d(z,z_2)\}$ for all $z\in D(\sigma,z_1,z_2)$. A $(1,Z)$-spaced pair is simply called a \textbf{spaced pair}.

Note that if $\sigma>1$, then $z_1,z_2$ can form a $(\sigma,D(\sigma,z_1,z_2))$-spaced pairs only if $D(\sigma,z_1,z_2)\subset Z\backslash\{z_1,z_2\}$.
\end{dfn}

\begin{crl}\label{GeodFailCrl}
Let $X$ be a geodesic space and $n\geq 3$. (i) The quasiconvexity constant of $FS_n(X)$ is at least $2$. (ii) $FS_n(X)$ is not geodesic (i.e., the quasiconvexity constant of $FS_n(X)$ exceeds $1$).
\end{crl}
\begin{proof}
Fix a real number $\varepsilon>0$. Pick disjoint sets $x=\{x_1,...,x_n\}$ and $y=\{y_1,...,y_n\}$ in $FS_n(X)\backslash FS_{n-1}(X)$, i.e., $x\cup y$ contains exactly $2n$ elements. Since $X$ is geodesic, we can choose the elements of $x$ and $y$ to satisfy the following arrangements (where the situation seems easiest to visualize if we let the said elements lie along a single geodesic path). Arrange the points in $x\cup y$ into sets $A_1,...,A_k\subset x\cup y=A_1\cup\cdots\cup A_k$ such that the following hold:
\begin{enumerate}[leftmargin=0.7cm]
\item\label{SpairItem1} $|A_i|=3$ for $1\leq i\leq s$, $|A_i|=2$ for $s+1\leq i\leq k$, and each $A_i$ contains at least one member from $x$ and at least one member from $y$ (i.e., the 2 or 3 members of each $A_i$ are mixed).
\item\label{SpairItem2} $\dist(A_i,A_j)>2\varepsilon$, for all $i\neq j$, $\diam(A_1)=...=\diam(A_s)=2\varepsilon$, $\diam(A_{s+1})=...=\diam(A_k)=\varepsilon$, and in each $A_i$ with $|A_i|=3$ the lone element from $x$ (resp., $y$) is a midpoint of the two elements that both come from $y$ (resp., $x$).
\end{enumerate}
Note that
\[
2n=|x|+|y|=|A_1|+\cdots+|A_k|=3s+2(k-s)=s+2k.
\]
By construction, $d_H(x,y)=\varepsilon$ and therefore, the only complete relation $R\subset x\times y$ that is proximal (as necessary by Theorem \ref{QGeodExistFS} for the existence of a geodesic between $x$ and $y$) has cardinality
\[
|R|=2s+(k-s)=s+k=2n-k\geq n+1~~\textrm{(if $s\geq 2$)}.
\]
This is because completeness of $R$ requires that in each $A_i$ of cardinality $3$ (i.e., $1\leq i\leq s$), the lone element from $x$ or $y$ must pair up with each of the other two elements (resulting in two elements of $R$), while in each $A_i$ of cardinality $2$ (i.e., $s+1\leq i\leq k$) we have exactly one pairing of its two elements (resulting in one element of $R$).

Fix $s\geq 2$. Then we also have $k\geq 2$ (since $n\geq 3$), and so $|R|\leq 2n-2$, that is,
\[
n+1\leq|R|\leq 2n-2~~\textrm{(where $|R|=2n-2$ holds only when $k=s=2$ and $n=3$)}.
\]
(\emph{Note}: The worst case equality $|R|=2n-2$ above might actually hold for $n>3$ if one instead rearranges $x\cup y$, say, into two groups $A=\{x_1,...,x_{n-1},y_n\}$ and $B=\{y_1,...,y_{n-1},x_n\}$ (where $x\cup y=A\cup B$) such that $\dist(A,B)>2\varepsilon$ and $\diam(A)=\diam(B)\leq 2\varepsilon$. By a similar argument as before, the only proximal complete relation here is $R=\{(x_i,y_n):i=1,...,n-1\}\cup\{(x_n,y_i):i=1,...,n-1\}$, which has cardinality $|R|=2n-2$. Depending on the dimension/structure of $X$, the elements in $A$ and $B$ could be further arranged to achieve certain desired results.)

{\flushleft \emph{Proof of (i)}}: The points $x,y$ above form a spaced pair in $FS_n(X)$, i.e., $\overline{N}_r(x)\cap\overline{N}_r(y)=\emptyset$ whenever $0<r<d_H(x,y)$, or equivalently, $d_H(x,y)\leq\max\{d_H(x,z),d_H(y,z)\}$, for all $z\in FS_n(X)$. To see this, observe that if $z\in\overline{N}_r(x)\cap\overline{N}_r(y)$ for some $0<r<d_H(x,y)$, then by the definition of Hausdorff distance, $A_1,...,A_s$ each neighbor at least two elements of $z$ and $A_{s+1},...,A_k$ each neighbor at least one element of $z$, giving $|z|\geq2s+(k-s)=|R|\geq n+1$ (i.e., $z\not\in FS_n(X)$). So, given any $\lambda$-quasigeodesic $\gamma:[0,1]\rightarrow FS_n(X)$ from $x$ to $y$, we have $d_H(x,y)\leq \max\{d_H(x,\gamma(1/2)),d_H(\gamma(1/2),y)\}\leq(\lambda/2)d_H(x,y)$, and so $~\lambda\geq 2$.

{\flushleft \emph{Proof of (ii)}:} This is an immediate consequence of (i). Alternatively, in the absence of (i), we can still directly give a proof of (ii) based on Theorem \ref{QGeodExistFS} as follows:

\emph{Alternative proof of (ii)}: Consider the spaced pair $x,y\in FS_n(X)$ constructed above (and the relation $R\subset x\times y$). We will show that, in view of Theorem \ref{QGeodExistFS}, any associated collection of paths $\left\{\gamma_{(a,b)}:(a,b)\in R\right\}$ in $X$ violates the necessity requirement ``$|\{\gamma_{(a,b)}(t):(a,b)\in R\}|\leq n$, for all $t\in[0,1]$'' of the theorem.

Let $A$ denote any of $A_1,...,A_s$, and assume without loss of generality (due to symmetry between $x$ and $y$) that $|A\cap x|=2$ and $|A\cap y|=1$, and let $A=\{x_{i_1},y_j,x_{i_2}\}$, where by hypotheses $y_j$ is a midpoint between $x_{i_1}$ and $x_{i_2}$ (i.e., $x_{i_1}$ and $x_{i_2}$ are opposite with respect to $y_j$). Then it is clear that these points satisfy $d(y_j,x_{i_1})=d(y_j,x_{i_2})=d_H(x,y)=\varepsilon$.
Therefore, the two possible component paths (of a geodesic in $FS_n(X)$ between $x$ and $y$), namely, $\gamma_1=\gamma_{(x_{i_1},y_j)}$ and $\gamma_2=\gamma_{(x_{i_2},y_j)}$ are necessarily geodesics in $X$ that satisfy the following: For $t,t'\in[0,1]$,
\[
d\left(\gamma_1(t),\gamma_1(t')\right)= d_H(x,y)|t-t'|~~\textrm{and}~~d\left(\gamma_2(t),\gamma_2(t')\right)= d_H(x,y)|t-t'|,
\]
and so also satisfy the inter-path distance bound
\[
\left|d\left(\gamma_1(t),\gamma_2(t)\right)-d\left(\gamma_1(t'),\gamma_2(t')\right)\right|\leq 2d_H(x,y)|t-t'|.
\]
If $\gamma_1$ and $\gamma_2$ joint up at some $t'=t^\ast\in(0,1)$, then $|d(\gamma_1(t),\gamma_2(t))-0|\leq 2d_H(x,y)|t-t^\ast|$, and so
\[
d(\gamma_1(0),\gamma_2(0))\leq 2d_H(x,y)t^\ast~~\textrm{and}~~d(\gamma_1(1),\gamma_2(1))\leq 2d_H(x,y)|1-t^\ast|,
\]
which is a contradiction since the midpoint/opposite locations of the endpoints of the paths imply
\[
d(\gamma_1(0),\gamma_2(0))>2d_H(x,y)t'~~\textrm{or}~~d(\gamma_1(1),\gamma_2(1))>2d_H(x,y)|1-t'|,~~\textrm{for any}~~t'\in(0,1).
\]
This shows it is impossible for $\gamma_1$ and $\gamma_2$ to join up into a single path. Hence, the necessity requirement ``$|\{\gamma_{(a,b)}(t):(a,b)\in R\}|\leq n$, for all $t\in[0,1]$'' of Theorem \ref{QGeodExistFS} cannot be satisfied.
\end{proof}

\begin{lmm}\label{GeodExistLmm2}
Let $X$ be a metric space and $x,y\in FS(X)$. There exists a proximal complete relation $R\subset x\times y$ such that $|R|\leq |x|+|y|$. Moreover, if $|x|\geq 2$ and $|y|\geq 2$, then we can choose $R$ such that $~|R|\leq|x|+|y|-2$.
\end{lmm}
\begin{proof}
Let $x\in X(n)\backslash X(n-1)$, $y\in X(m)\backslash X(m-1)$. By the definition of $d_H(x,y)$, for all $i\in\{1,...,n\}$ and $j\in\{1,...,m\}$, there exist $\alpha(i)\in\{1,...,m\}$ and $\beta(j)\in\{1,...,n\}$ such that $d(x_i,y_{\alpha(i)})\leq d_H(x,y)$ and $d(x_{\beta(j)},y_j)\leq d_H(x,y)$. Let
$R=\{(x_i,y_{\alpha(i)}):i=1,...,n\}\cup\{(x_{\beta(j)},y_j):j=1,...,m\}$. Then $|R|\leq n+m$.

Moreover, if $n\geq 2$ and $m\geq 2$, then with $d(x_{i_0},y_{\alpha(i_0)})=d_H(x,y)$ and $d(x_{\beta(j_0)},y_{j_0})=d_H(x,y)$, we see that $R$ contains both $(x_{i_0},y_{\alpha(i_0)})$ and $(x_{\beta(\alpha(i_0))},y_{\alpha(i_0)})$, and by a symmetric argument, $R$ also contains both $(x_{\beta(j_0)},y_{j_0})$ and $(x_{\beta(j_0)},y_{\alpha(\beta(j_0))})$). Removing the two redundant elements, we are left with another complete relation $\widetilde{R}=R\backslash\{(x_{\beta(\alpha(i_0))},y_{\alpha(i_0)}),(x_{\beta(j_0)},y_{\alpha(\beta(j_0))})\}$ satisfying $|\widetilde{R}|=|R|-2\leq m+n-2$.
\end{proof}

\begin{crl}\label{FSGeodBound}
Let $X$ be a metric space. (i) Let $n\geq 2$. If $X$ is geodesic and $m\geq 2n-2$, then any two points of $FS_n(X)$ can be connected by a geodesic in $FS_m(X)$. (ii) $X$ is geodesic if and only if $FS_2(X)$ is geodesic. (iii) If $X$ is geodesic, then so is $FS(X)$ (also by Corollary \ref{FS_BCl_QConv}(ii)).
\end{crl}{\tiny {\scriptsize }}
\begin{proof}
(i): If $m\geq 2n-2$, then for any $x,y\in FS_n(X)$, there exists (by Lemma \ref{GeodExistLmm2}) a proximal complete relation $R\subset x\times y$ such that $|R|\leq |x|+|y|-2\leq 2n-2\leq m$, and so (by Theorem \ref{QGeodSuffFS}) $x$ and $y$ are connected by a geodesic in $FS_m(X)$. (ii) and (iii) immediately follow from (i).
\end{proof}

\begin{rmk}\label{CEx4Rmk}
Even when $X$ is geodesic, if $n\geq 2$, then $FS^n(X):=FS(X)\backslash FS_{n-1}(X)$ need not be geodesic. For a counterexample, let $X$ be the union of two coordinate axes in $\mathbb{R}^2$ with the taxicab metric, that is, $X = \{(x, y) \in  \mathbb{R}^2 : xy = 0\}$ and $d((x_1, y_1)$, $(x_2, y_2)) = |x_1 - x_2| + |y_1 - y_2|$. This is a geodesic space, since the distance is precisely the length of the natural path from $(x_1, y_1)$ to $(x_2, y_2)$. Let $A = \{(1, 0), (0, 1)\}$ and $B = \{(-1, 0), (0,-1)\}$, which are elements of $FS^2(X)$. If there is a geodesic $\gamma : [0, 1] \rightarrow FS^2(X)$ connecting $A$ to $B$, then (by Note \ref{QGeodNote}) $\gamma (1/2)\subset \overline{N}_1(A) \cap \overline{N}_1(B) = \{(0, 0)\}$, which is a contradiction (since one-point sets are not in $FS^2(X)$).
\end{rmk}

The proof of Theorem \ref{QConvFSn} (which is our main result of this section) uses Lemma \ref{FactznProp} and the meaning of a \emph{reduced complete relation} from Definition \ref{RelTypeDfn}.

\begin{thm}[\textcolor{blue}{Quasiconvexity of $FS_n(X)$}]\label{QConvFSn}
If $X$ is geodesic, then $FS_n(X)$ is $2$-quasiconvex.

Moreover, if $n\geq 3$, then by Corollary \ref{GeodFailCrl},  the quasiconvexity constant $2$ for $FS_n(X)$ is the least.
\end{thm}
\begin{proof}
Fix $x,y\in FS_n(X)$. Let $R\subset x\times y$ be a reduced proximal complete relation (which is possible because a proximal complete relation $R\subset x\times y$ exists by the definition of $d_H(x,y)$ and can be reduced by removing a finite number of inessential elements, i.e., those $(a,b)\in R$ that satisfy $|(\{a\}\times B)\cap R|\geq 2$ and $|(A\times\{b\})\cap R|\geq 2$).

Consider the domain $x^1=\{a\in x: |(\{a\}\times B)\cap R|=1\}$ and range $y^1=\{b\in y: |(A\times\{b\})\cap R|=1\}$ of $R$. Define a map $f:x^1\rightarrow y$ by $(\{a\}\times B)\cap R=\{(a,f(a))\}$ and a map $g:y^1\rightarrow x$ by $(A\times\{b\})\cap R=\{(g(b),b)\}$. Then
\[
R=\{(a,f(a)):a\in x^1\}\cup\{(g(b),b):b\in y^1\},
\]
where $\{(a,f(a)):a\in x^1\}\cap\{(g(b),b):b\in y^1\}=\{(a,f(a)):a\in x^0\}=\{(g(b),b):b\in y^0\}$ for subsets $x^0\subset x^1$ and $y^0\subset y^1$ such that $f^0=f|_{x^0}:x^0\rightarrow y^0$ and $g^0=g|_{y^0}:y^0\rightarrow x^0$ are mutually inverse bijections. So, with $x':=x^1$, $y':=y^1\backslash y^0$, $x'':=g(y')$, and $y'':=f(x')$, we get (``parallel'' or ``disjoint'') surjective maps
\[
f'=f|_{x'}:x'\rightarrow y'',~~~~g'=g|_{y'}:y'\rightarrow x''
\]
and disjoint unions $~x=x'\sqcup x''=x'\sqcup g'(x')$, $~y=y'\sqcup y''=y'\sqcup f'(x')$.

Let $z:=x''\cup y''$, where $z\in FS_n(X)$ by construction since $|z|\leq|x''|+|y''|\leq|x''|+|x'|=|x|\leq n$. Consider the $d_H (x,y)/d_H (x,z)$-proximal complete relation $R_1\subset x\times z$ and the $d_H(x,y)/d_H(z,y)$-proximal complete relation $R_2\subset y\times z$ given respectively by
\[
R_1=\{(a,f'(a)):a\in x'\}\cup\{(c,c):c\in x''\}~~\textrm{and}~~R_2=\{(g'(b),b):b\in y'\}\cup\{(c,c):c\in y''\},
\]
where the proximality claims for $R_1$ and $R_2$ are due (by Lemma \ref{FactznProp} and proximality of $R$) to $d_H(x,z)=d_H(x'\cup x'',x''\cup y'')\leq d_H(x',y'')\leq d_H(x,y)$ (along with $d(a,f'(a))\leq d_H(x,y)$) and, similarly, $d_H(z,y)\leq d_H(x,y)$ (along with $d(g'(b),b)\leq d_H(x,y)$). Then, by Theorem \ref{QGeodSuffFS}, we get a $\lambda_1:=d_H(x,y)/d_H(x,z)$-quasigeodesic $\gamma_1:[0,1]\rightarrow FS_n(X)$ from $x$ to $z$ and a $\lambda_2:=d_H(x,y)/d_H(z,y)$-quasigeodesic $\gamma_2:[0,1]\rightarrow FS_n(X)$ from $z$ to $y$. The path $\gamma=\gamma_1\cdot\gamma_2:[0,1]\rightarrow FS_n(X)$ from $x$ to $y$ given by $\gamma|_{[0,1/2]}(t):=\gamma_1(2t)$ and $\gamma|_{[1/2,1]}(t):=\gamma_2(2t-1)$ satisfies
\[
\ell(\gamma)=\ell(\gamma_1)+\ell(\gamma_2)\leq \lambda_1d_H(x,z)+\lambda_2d_H(z,y)=2d_H(x,y). \qedhere
\]
\end{proof}

\begin{crl}[\textcolor{blue}{\cite[Corollary 2.1.15]{akofor2020}}]\label{FinalCor}
If $X$ is $\lambda$-quasiconvex, then $FS_n(X)$ is $\alpha_n(\lambda)$-quasiconvex with the smallest constants satisfying $\alpha_1(\lambda)=\alpha_2(\lambda)=\lambda$ and $\max(2,\lambda)\leq\alpha_n(\lambda)\leq 2\lambda$, for $n\geq 3$.
\end{crl}
\begin{proof}
Suppose $X$ is $\lambda$-quasiconvex. Let $\alpha_n(\lambda)$ denote the smallest quasiconvexity constant of $FS_n(X)$. Since $X$ is an $F_1$-space (hence $FS_1(X)$ is isometric to $X$), $FS_1(X)$ is $\lambda$-quasiconvex, and so $\alpha_1(\lambda)=\lambda$. By replacing ``geodesic (path)'' with ``$\lambda$-quasiconvex (path)'' in both Theorem \ref{QGeodSuffFS} and Corollary \ref{FSGeodBound}, we see that $FS_2(X)$ is also $\lambda$-quasiconvex, and so $\alpha_2(\lambda)=\lambda$.

For all $n\geq 1$, since $X\subset FS_n(X)$ (with meaning in Notation \ref{TbPcTerm}), we have $\alpha_n(\lambda)\geq \lambda$. Finally, by replacing ``geodesic'' with ``$\lambda$-quasiconvex'' (hence rescaling the Lipschitz constant of involved paths by $\lambda$) in Theorem \ref{QConvFSn}, we obtain a straightforward analog of Theorem 4.8 which says that $FS_n(X)$ is $2\lambda$-quasiconvex (hence $\alpha_n(\lambda)\leq 2\lambda$) and, for $n\geq 3$, the analog of Corollary \ref{GeodFailCrl}(i) still says the quasiconvexity constant of $FS_n(X)$ is at least $2$ (hence $\alpha_n(\lambda)\geq 2$).
\end{proof}

\section{\textnormal{\bf Some relevant questions}}\label{QGeodQuests}

\noindent Fix a metric space $X$. According to Theorem \ref{RepThmPCl}, to have a Lipschitz path in a stable covering subspace $\mathcal{J}\subset PCl(X)$ it is necessary and sufficient to have a set of Lipschitz paths in $\widetilde{X}$ with specific properties. It is clear that the sufficiency (with proof based on Lemma \ref{QGeodSuff}) holds for $\mathcal{J}\subset BCl(X)$, and not just for $\mathcal{J}\subset PCl(X)$. The necessity (which depends on Theorem \ref{QGeodComp}), however, is more involved.

\vspace{0.2cm}
\begin{question}\label{CritQuest}
In Theorem \ref{RepThmPCl}, can $PCl(X)$ be replaced with $BCl(X)$? Alternatively:
\begin{enumerate}
\item For the existence of Lipschitz paths in a stable covering subspace $\mathcal{J}\subset BCl(X)$, is the sufficient condition in Theorem \ref{RepThmPCl} (which is valid by Lemma \ref{QGeodSuff}) also necessary?
\item Does there exist a nontrivial Lipschitz path in a stable covering subspace $\mathcal{J}\subset BCl(X)$ that violates the sufficient condition in Theorem \ref{RepThmPCl}?
\end{enumerate}
If the answer to (1) above is negative (i.e., the answer to (2) above is positive), what is the largest possible subspace of $BCl(X)$ for which the necessity part of Theorem \ref{RepThmPCl} is valid?
\end{question}

The answer to Question \ref{CritQuest} might require or involve the answer to the following related question.
\begin{question}\label{RepsQuest}
In Proposition \ref{QGeodEquiv2}, can the canonical map $\gamma(t)=A_{t{L}}\cap B_{(1-t){L}}$ in $BCl(X)$ be replaced with an arbitrary $L$-Lipschitz path in $BCl(X)$?
\end{question}

For application purposes, one can also ask questions concerning efficiency in practically constructing or realizing Lipschitz paths in subset spaces. Consider the setup in Theorem \ref{RepThmPCl}.

\begin{question}\label{CardQuest}
Among the possible densely complete relations $R=\{(\gamma_r(0),\gamma_r(1)):r\in\Gamma\}\subset \widetilde{A}\times \widetilde{B}$, which ones, and how many of them, have the smallest cardinality (in the sense that they best approximate reduced densely complete relations)?
\end{question}

\begin{question}\label{LengthQuest}
Among the possible (reduced) densely complete relations $R=\{(\gamma_r(0),\gamma_r(1)):r\in\Gamma\}\subset \widetilde{A}\times \widetilde{B}$, which ones, and how many of them, admit the shortest, simplest, or least (computationally) complex paths $\{\gamma_r:r\in \Gamma\}$ in $\widetilde{X}$?
\end{question}

The Castaing representation of Lipschitz compact-valued mappings (\cite[Th\'eor\`eme 5.4]{castaing1967} and \cite[Theorem 7.1]{Chist2004}), in terms of only a countable number of Lipschitz selections, might be relevant to Questions \ref{CardQuest} and \ref{LengthQuest}.

\begin{question}
In Corollary \ref{FinalCor} (where $X$ is $\lambda$-quasiconvex), we know that $\alpha_1(\lambda)=\alpha_2(\lambda)=\lambda$. What precisely is $\alpha_n(\lambda)$ for $n\geq 3$? For example, is $\alpha_n(\lambda)=2\lambda$ for $n\geq 3$?

Observe that if we can find $(\lambda,D(\lambda,z_1,z_2))$-spaced pairs in $FS_n(X)$, then by an argument similar to that in the proof of Corollary \ref{GeodFailCrl}(i), we might get $\alpha_n(\lambda)\geq 2\lambda$ and conclude that $\alpha_n(\lambda)=2\lambda$. Therefore, as a followup question, do there exist any $(\lambda,D(\lambda,z_1,z_2))$-spaced pairs in $FS_n(X)$?
\end{question}

\section*{\textnormal{\bf Acknowledgments}}

This manuscript has been greatly improved using extensive feedback (including the counterexamples in Remarks \ref{IndexSetrmk} and \ref{WeakConnRmk}) from referees and editors of \emph{Annales Fennici Mathematici}. The same is true about careful feedback (including the counterexamples in Remarks \ref{CEx3Rmk} and \ref{CEx4Rmk}, the fact in Remark \ref{CEx3Rmk} that $PCl(X)$ is $d_H$-closed, and attention to the references \cite{BelovChist2000}, \cite{castaing1967}, \cite{Chist2004}, \cite{hermes1971}, \cite{Slezak1987}) from referees and editors of \emph{The Journal of Analysis}.

\section*{\textnormal{\bf Declarations}}
The author has no competing interests to declare that are relevant to the content of this article.


\vspace{0.2cm}
\hrule
\endgroup
\end{document}